\documentclass[10pt]{article} 

\usepackage[preprint]{tmlr}

\usepackage{fancyhdr}
\usepackage{cancel}
\usepackage{amsthm, amssymb, amsmath}
\usepackage{mathtools}
\usepackage{color}
\usepackage{tikz}
\usetikzlibrary{arrows}
\usepackage[utf8]{inputenc}
\usepackage[shortlabels]{enumitem}
\usepackage{url}
\usepackage{alltt}
\usepackage{memhfixc}
\usepackage{bbold}%
\usepackage{cases}
\usepackage{comment}
\usepackage{tabularx}
\usepackage{eucal}
 
\definecolor{labelkey}{rgb}{0.0, 0.8, 0.3}
\usepackage[top=2.5cm,bottom=2.5cm,left=2.5cm,right=2.5cm,includeheadfoot,asymmetric]{geometry}
\usepackage{algorithm}
\usepackage{wrapfig}
\usepackage{natbib}

\allowdisplaybreaks

\numberwithin{equation}{section}
\definecolor{bbblue}{rgb}{0.4, 0.4, .9}

\usepackage{hyperref}
\hypersetup{
  colorlinks = true,
  urlcolor = blue, 
  linkcolor = blue,
  citecolor = blue}

\definecolor{ggreen}{rgb}{0.0, 0.5, 0.3}
\definecolor{rred}{rgb}{0.65, 0.2, 0.2}
\definecolor{bblue}{rgb}{0.0, 0.0, 1}

\newtheorem{thm}{Theorem}
\newtheorem*{thm*}{Theorem}
\newtheorem*{ass*}{Assumption}
\newtheorem{assump}{Assumption}
\newtheorem{rem}[thm]{Remark}

\newtheorem{lem}[thm]{Lemma}
\newtheorem{cor}[thm]{Corollary}
\newtheorem{pro}[thm]{Proposition}

\newcommand{\bX}{\boldsymbol{X}} 
\newcommand{\bx}{\boldsymbol{x}} 
\newcommand{\bS}{\boldsymbol{S}} 
\newcommand{\bA}{\boldsymbol{A}} 
 
\newcommand{\bSigma}{\boldsymbol{\Sigma}} 
\newcommand{\bz}{\boldsymbol{z}} 
\newcommand{\by}{\boldsymbol{y}} 
\newcommand{\bw}{\boldsymbol{w}} 
\newcommand{\bzero}{\boldsymbol{0}} 
\newcommand{\bU}{\boldsymbol{U}} 
\newcommand{\bV}{\boldsymbol{V}} 
\newcommand{\bI}{\boldsymbol{I}} 
\newcommand{\bC}{\boldsymbol{C}} 
\newcommand{\bL}{\boldsymbol{L}} 
\newcommand{\bG}{\boldsymbol{G}} 
\newcommand{\bDelta}{\boldsymbol{\Delta}} 
\newcommand{\bu}{\boldsymbol{u}} 
\newcommand{\bv}{\boldsymbol{v}} 
\newcommand{\bT}{\boldsymbol{T}} 
\newcommand{\bR}{\boldsymbol{R}} 
\newcommand{\bP}{\boldsymbol{P}} 
\newcommand{\bJ}{\boldsymbol{J}} 
\newcommand{\bF}{\boldsymbol{F}} 

\newcommand{\cA}{\mathcal{A}}
\newcommand{\cC}{\mathcal{C}}
\newcommand{\cS}{\mathcal{S}} 
\newcommand{\cN}{\mathcal{N}}

\newcommand{\cB}{\mathcal{B}} 

\newcommand{\bbS}{\mathbb{S}}

\newcommand{\Sp}{\mathrm{Sp}} 		

\newcommand{\E}{\mathbb{E}}

\newcommand{\cP}{\mathcal{P}}
\newcommand{\dbw}{d_{\mathrm{BW}}}

\newcommand{\N}{\mathbb N}
\newcommand{\R}{\mathbb R}

\newcommand{\argmin}{\mathop{\mathrm{argmin}}}

\renewcommand{\phi}{\varphi}

\DeclareMathOperator{\Tr}{Tr}

\DeclareMathOperator{\id}{id}

\DeclareMathOperator{\rank}{rank}

\newcommand{\bSigmabar}{\overline{\bSigma}}
\newcommand{\xxt}{\bx\bx^\top}
\newcommand{\wwt}{\bw\bw^\top}

\newcommand{\cNo}{\cN_0(\R^d)}

\newcommand{\sign}{\mathrm{sign}}
\newcommand{\diag}{\mathrm{diag}}
\newcommand{\BW}{\textsf{BW}}

\newcommand{\mbb}[1]{\mathbb{#1}}

\begin{document}

\title{Bures-Wasserstein Barycenters and Low-Rank Matrix Recovery}

\author{\name Tyler Maunu \email maunu@brandeis.edu \\
      \addr Brandeis University
      \AND
      \name Thibaut Le Gouic \email thibaut.le\_gouic@math.cnrs.fr \\
      \addr \'Ecole Centrale de Marseille
      \AND
      \name Philippe Rigollet \email rigollet@math.mit.edu\\
      \addr MIT}


\maketitle

\thispagestyle{empty}





%
%
%
%
%
%
%
%


\begin{abstract}
	We revisit the problem of recovering a low-rank   positive semidefinite matrix from rank-one projections using tools from optimal transport. More specifically, we show that a variational formulation of this problem is equivalent to computing a Wasserstein barycenter. In turn, this new perspective enables the development of new geometric first-order methods with strong convergence guarantees in Bures-Wasserstein distance. Experiments on simulated data demonstrate the advantages of our new methodology over existing methods.
\end{abstract}


\section{Introduction}

Recovering a low-rank matrix is a fundamental primitive across many settings, such as matrix completion \citep{fazel2002matrix,candes2009exact,candes2010power}, phase retrieval~\citep{candes2015phase}, principal component analysis \citep{pearson1901liii,hotelling1933analysis}, robust subspace recovery \citep{lerman2018overview}, and robust principal component analysis \citep{chandrasekaran2009sparse,candes2011robust,xu2010robust}. This line of work can be understood as a generalization of the classical compressed sensing question \citep{donoho2006compressed, candes2006robust}, where the goal is the recovery of a sparse vector. This problem can be cast as a low-rank recovery problem over diagonal matrices. In all of these settings, the assumption of a low-rank structure is essential for efficient estimation and optimization in high-dimensional settings.

While the above applications all aim at recovering a low-rank matrix $\bS$, the observational---a.k.a sensing---mechanism that governs access to $\bS$ comes in many declinations. For the purpose of applications, it is often sufficient to focus on linear measurements of the form $\langle \bS, \bA\rangle$ for some given sensing matrix $\bA$. This setup covers a wide variety of applications ranging from covariance sketching \citep{chen2015exact} and low-rank matrix completion \citep{candes2010matrix, recht2010guaranteed} to phase retrieval \citep{fienup1978reconstruction,candes2013phaselift,candes2015phase} and quantum state tomography \citep{gross2010quantum}. New solutions to this problem can have many practical implications. 

 In this paper, we focus on a specific instantiation of this problem, where the measurement matrix $\bA=\bx\bx^\top$ is rank-one and positive semidefinite (PSD) so that $\langle \bx\bx^\top, \bS\rangle=\bx^\top \bS \bx$. This important case of the low-rank matrix recovery problem  
has received significant attention over the past few years \citep{cai2015rop,chen2015exact,sanghavi2017local,li2019nonconvex}. 

Assume that we observe 
\begin{equation}\label{eq:matrecmeas}
    y_i = \bx_i^\top  \bS \bx_i\,, i=1, \ldots, n
\end{equation}
where $\bS \in \R^{d \times d}$ is an unknown rank $r$ PSD matrix and $\bx_1, \ldots, \bx_n$ are i.i.d from some distribution.
Our goal is to recover or estimate $\bS$ from the pairs $(y_i, \bx_i)$, $i=1, \ldots, n$. Throughout, we  denote by $\mbb S_{+}^d$ the set of $d \times d$ PSD matrices and $\mbb S_{++}^d$ is the set of positive definite (PD) matrices.

Finding a low-rank matrix $\bS$ subject to constraints~\eqref{eq:matrecmeas} is a semidefinite program (SDP) that can be implemented in polynomial time using general-purpose solvers. Furthermore, the specific structure of this SDP may be leveraged to derive faster algorithms. Such solutions include the Burer-Monteiro approach to solving semidefinite programs \citep{burer2003nonlinear} or nonconvex gradient descent methods for low-rank programs \citep{sanghavi2017local}.  Often, these approaches result in \emph{nonconvex} optimization programs for which theoretical results are limited. 


In this work, we take a principled approach to solving this problem by eliciting convexity using a specific geometry on the the space of PSD matrices. More precisely, we employ the Bures-Wasserstein (hereafter \BW) geometry, which comes independently from optimal transport and quantum information theory~\citep{bures1969extension, bhatia2019bures}. This geometry allows us to solve the original problem  by computing a  \emph{\BW\ barycenter} \citep{agueh2011barycenter,alvarez2016fixed,chewi2020gradient,altschuler2021averaging}. In turn, we employ  geodesic gradient descent on the \BW\ manifold to compute said barycenter. We propose both full gradient and stochastic gradient based methods that are guaranteed to efficiently recover a low-rank matrix. These methods have low computational cost (per iteration complexity of $O(ndr)$ for gradient descent and $O(dr)$ for stochastic gradient descent), have minimal parameter tuning, are easily implemented, and show excellent practical performance. We demonstrate an example application of phase retrieval in Figure \ref{fig:pr}. In this set-up, \BW\ gradient descent recovers the image faster than Wirtinger Flow (WF) \citep{candes2015phase}, and \BW\ gradient descent needs no parameter tuning.

\begin{figure}[h!]
    \centering
    \includegraphics[width = .8\columnwidth]{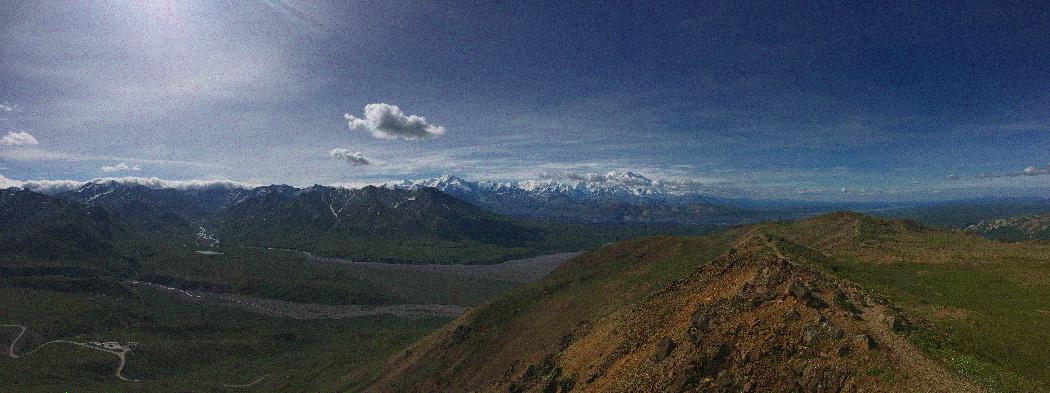}
    \includegraphics[width = .8\columnwidth]{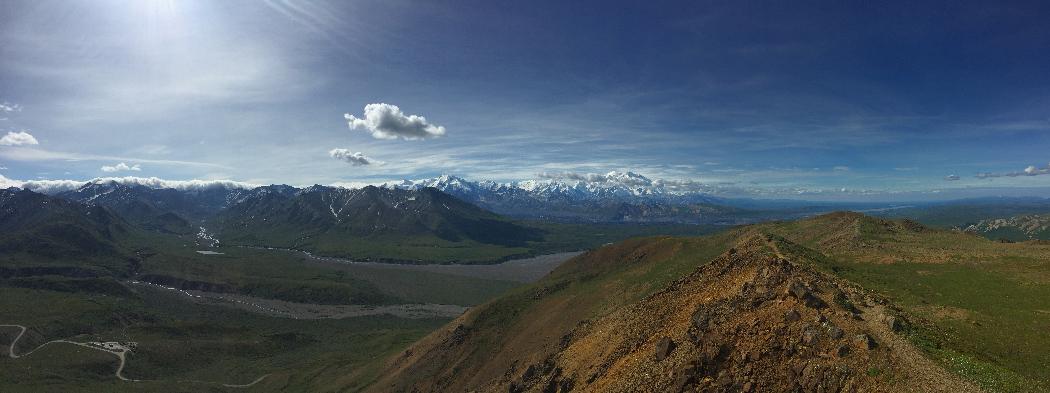}
    \caption{Recovered images after 140 iterations of Wirtinger Flow~\citep{candes2015phase} and \BW\ gradient descent (our method). \BW\ gradient descent recovers a much sharper image within the same number of iterations; note that the iteration complexity of both methods is the same.}
    \label{fig:pr}
\end{figure}



\medskip

\noindent{\bf Main contributions.} The main results of this paper are:
\begin{enumerate}
    \item We prove that the barycenter of a certain distribution of rank-one Gaussians exactly recovers the underlying low-rank matrix.
	\item With this connection, we give novel geodesic gradient descent and stochastic geodesic gradient descent algorithms for solving the low-rank PSD matrix recovery problem using existing first-order algorithms for computing \BW\ barycenters. 
	\item Existing first-order algorithms for computing \BW\ barycenters are only guaranteed to work for full rank distributions. Since our method considers barycenters of rank-one PSD matrices, we develop new theory and give a guarantee of local linear convergence in \BW\ distance for the gradient descent method. We also discuss initialization of our method.
	\item We demonstrate the competitive edge of our algorithms in a few experimental settings.
\end{enumerate}

\medskip
\noindent{\bf Related Work.}
Many methods have been proposed to solve variants of the matrix recovery problem. Original ideas for this problem trace back to linear systems theory, low-rank matrix completion, low-dimensional Euclidean embeddings, and image compression \citep{recht2010guaranteed}.


We focus here on rank-one projections of positive semidefinite matrices as in~\eqref{eq:matrecmeas}. This setup is either specifically considered or a special case of a large number of works including~\citep{candes2015phase,cai2015rop,chen2015exact,zhong2015efficient,wang2016solving,wang2017solving,sanghavi2017local,li2019nonconvex}. These methods can be clustered into two families. The first one aims at minimizing convex relaxations of an energy functional that often based on the nuclear norm~\citep{cai2015rop,chen2015exact}. Such convex programs can be solved via standard solvers. Another family of methods directly implement the low-rank constraint into a nonconvex constraint~\citep{li2019nonconvex} thus manipulating candidate matrices with smaller representations and thereby boosting computational efficiency; see~\cite{chi2019nonconvex} for an overview of such algorithms.
The special case where $\bS$ has rank one corresponds to the classical phase retrieval problem and has received much attention with dedicated algorithms~\citep{fienup1978reconstruction,candes2015phase,wang2016solving,wang2017solving,chi2019nonconvex}.

Note that the methods introduced in the present paper can be readily extended to the \emph{covariance recovery}  problem framed in \cite{cai2015rop}, and that is similar to estimation in a random effects model. Under this model, rather than~\eqref{eq:matrecmeas}, we observe $y_i = \langle \bx_i, \bw_i\rangle + \epsilon_i,$
 where $\bw_i \sim N(\bzero, \bS)$, where $N(\bzero, \bS)$ is the centered Gaussian distribution on $\R^d$ with PSD covariance matrix $\bS$. The goal here is to recover the covariance matrix $\bS$ of the weight vectors from observations of $(y_i, \bx_i)$, $i=1, \ldots, n$. This problem has natural connections to mixture of regressions models as well~\citep{de1989mixtures,yi2014alternating,zhong2016mixed,sedghi2016provable}.


In terms of the complexity of various methods, solving the semidefinite program using off-the-shelf solvers takes $O(nd^2 + d^3)$ complexity.
Gradient descent on PSD matrices (see the initialization phase in \citep{tu2016low}) can solve this with per-iteration complexity $O(nd^2)$, and one can prove linear convergence under certain assumptions. 
The most directly comparable methods are nonconvex gradient descent \citep{li2019nonconvex}, which have per-iteration complexity $O(ndr)$ and utilize a Burer-Monteiro factorization~\citep{burer2003nonlinear}. As we will see, our method also has per-iteration complexity of $O(ndr)$.

Finally, we mention several works dedicated to the computation of the nonconvex Wasserstein barycenter problem \citep{agueh2011barycenter,alvarez2016fixed,zemel2019procrustes,chewi2020gradient, altschuler2021averaging}. No theoretical study in these works allows for rank deficiency.

\medskip
\noindent{\bf Notation.}
Bold capital letters  denote matrices while bold lower-case letters  denote vectors. The Hilbert-Schmidt inner product is $\langle\cdot, \cdot \rangle$, and  $\langle \cdot, \cdot \rangle_{\gamma_{\bSigma}}  = \langle \cdot, \cdot \rangle_{\bSigma}$ is the Riemannian metric associated to the (fixed-rank) \BW\ manifold at $\bSigma$. Their corresponding norms are written as $\|\cdot\|$ and $\|\cdot\|_{\bSigma}$, respectively. The orthogonal projection onto the column span of $\bSigma$ is $\bP_{\bSigma} = \bP_{\Sp(\bSigma)}$. Similarly, $\bP_{\bSigma}^\perp$ is the projection onto null-space of $\bSigma$ (orthogonal complement of $\Sp(\bSigma)$). The Dirac distribution at a point $x$ denoted by $\delta_{x}$.

\medskip
\noindent{\bf Outline.}
We begin by outlining our approach to matrix recovery in Section \ref{sec:approach}.
We then give the main theoretical results for our approach in Section~\ref{sec:theory}. After this, we give experiments demonstrating these advantages in Section~\ref{sec:exp}, and discuss the limitations of our work in Section \ref{sec:lim}.


\section{The Bures-Wasserstein Barycenter Approach}
\label{sec:approach}

Recall that we aim at find the rank $r$ matrix $\bS\in \mbb S_+^d$  given the observations \eqref{eq:matrecmeas}. We will use the notation $y_i = \cA(\bS)_i= \langle \bx_i \bx_i^\top,\bS\rangle $, $i=1,\dots,n$.
To recover a low-rank matrix $\bS$ from these measurements, most past work has focused on some form of energy minimization. For example, some works have looked at convex nuclear norm minimization methods. \cite{cai2015rop} and \cite{chen2015exact} concurrently developed a nuclear norm minimization procedure that solves
\begin{equation}\label{eq:nucl}
	\min_{\substack{\bSigma \in \bbS_{+}^d ,\ \cA(\bSigma) = \by }} \Tr(\bSigma),
 \end{equation}
where $\by = [y_1, \dots, y_n]^\top$ and $\cA(\bSigma) = [\cA(\bSigma)_1, \dots, \cA(\bSigma)_n]^\top$. 
One can directly solve the semidefinite program using standard convex optimization packages.
A Lagrangian formulation of \eqref{eq:nucl} yields the energy
$\frac{1}{2} \|\cA(\bSigma) - \by \|_2^2 + \lambda \Tr(\bSigma),$
which can also be minimized using a variety of methods.

In the following, we lay out our approach to the low-rank matrix recovery problem, which focuses on a new energy minimization procedure. 
In Section~\ref{subsec:nonconvmr} we discuss the common nonconvex approaches to matrix recovery and outline our novel optimization program. Then, in Section~\ref{subsec:bwbprob}, we discuss the \BW\ barycenter problem, and show how it recovers solutions to the energy minimization we propose. After this, in Section \ref{subsec:algos} we outline the first-order algorithms for computing \BW\ barycenters. We finish in Section \ref{subsec:regalgo} by discussing a regularization procedure that allows one to estimate higher rank proxies, from which it is possible to recover $\bS$.

\subsection{Nonconvex Approaches for Matrix Recovery}
\label{subsec:nonconvmr}

Suppose that we know an upper bound for the rank of the underlying matrix $\bS$. We could utilize this information in a nonconvex optimization program such as
\begin{align}\label{eq:lsmatrec_nc1}
	\min_{\substack{\bSigma \in \bbS_{+}^d ,\ \rank(\bSigma) \leq r}}\ &\frac{1}{2n} \|\cA(\bSigma) - y\|_2^2  .
\end{align}
Without the rank restriction (i.e., $r=d$) this problem is in fact convex. For any fixed $r \leq d$, we can parameterize the rank $r$ matrices in $\bbS_+^{d}$ by $\bU \bU^\top$, for $\bU \in \R^{d \times r}$, which is now commonly referred to as Burer-Monteiro factorization~\citep{burer2003nonlinear}.  We thus define the set of PSD matrices of rank at most $r$ using this factorization: $\bbS_+^{d,r} := \{\bSigma \in \bbS_+^d : \bSigma = \bU\bU^\top, \ \bU \in \R^{d \times r}\}$. With this parametrization, the matrix recovery problem in \eqref{eq:lsmatrec_nc1} is equivalent to
\begin{equation}\label{eq:lsmatrec_nc2}
	\min_{\substack{\bU \in \R^{d \times r}  }} \frac{1}{2} \|\cA(\bU \bU^\top) - \by\|_2^2.
\end{equation}

While past work has focused on these least squares formulations, there have not been many modifications of this energy. We propose the following modifications to the energies \eqref{eq:lsmatrec_nc1} and \eqref{eq:lsmatrec_nc2}:
\begin{align}\label{eq:lsmatrec_new}
	\min_{\substack{\bSigma \in \bbS_{+}^d, \ \rank(\bSigma) \leq r}}\  &\frac{1}{2n} \| \sqrt{\cA(\bSigma)} - \sqrt{\by}\|_2^2  =\min_{\bU \in \R^{d \times r}}  \frac{1}{2n} \| \sqrt{\cA(\bU\bU^\top)} - \sqrt{\by}\|_2^2,
\end{align}
where the square root is taken componentwise. As we demonstrate in the following sections, this problem has a natural solution as a \BW\ barycenter. 


\subsection{The Bures-Wasserstein Barycenter Problem}
\label{subsec:bwbprob}


To explain the connection of \eqref{eq:lsmatrec_new} to \BW\ barycenters, we will first explain how \BW\ space arises from the perspective of optimal transport~\citep{villani2009optimal}. Let $\cP_2(\R^d)$ be the set of all measures on $\R^d$ with finite second moment. The 2-Wasserstein distance between measures $\mu$ and $\nu \in \cP_2(\R^d)$ is defined by
\begin{equation}\label{eq:2wass}
    W_2^2(\mu, \nu) = \inf_{\pi \in \Pi(\mu, \nu)} \E_{(\bx,\by) \sim \pi} \|\bx - \by\|^2,
\end{equation}
where $\Pi(\mu, \nu)$ denotes the set of all couplings between $\mu$ and $\nu$ (i.e., the set of all joint distributions on $\R^d \times \R^d$ with marginals $\mu$ and $\nu$). 
The 2-Wasserstein distance defines a metric over $\cP_2(\R^d)$, and the resulting geodesic metric space is referred to as 2-Wasserstein space.

Let $\cN(\R^d)$ denote the set of Gaussian distributions on $\R^d$, and $\cN_0(\R^d)$ be the set of centered Gaussian distributions.
Both are geodesically weakly convex subsets of 2-Wasserstein space, meaning there always exist 2-Wasserstein geodesics between points in these sets that are contained within these sets. Letting $N(\bzero, \bSigma) \in \cNo$ denote the Gaussian distribution on $\R^d$ with mean zero and covariance matrix $\bSigma \in \bbS_+^d$, the 2-Wasserstein distance between $N(\bzero, \bSigma_0)$ and $N(\bzero, \bSigma_1)\in \cNo$ has the explicit form 
\begin{align}\label{eq:GaussWass}
	W_2^2&(N(\bzero, \bSigma_0), N(\bzero, \bSigma_1)) = \Tr\big[\bSigma_0 + \bSigma_1 - 2 (\bSigma_0^{1/2}  \bSigma_1 \bSigma_0^{1/2})^{1/2}\big].
\end{align}
Notice that this is purely a function of the covariance matrices, and so the Wasserstein distance induces a distance metric on PSD matrices called the \emph{Bures-Wasserstein distance} \citep{bhatia2019bures}.
To refer to this distance over PSD matrices rather than the Gaussian distributions, we will write
\begin{equation}
	\dbw (\bSigma_0, \bSigma_1) = W_2(N(\bzero, \bSigma_0), N(\bzero, \bSigma_1)).
\end{equation}
More than just giving the set of PSD matrices a distance metric, this identification endows $\bbS_+^d$ with a natural Riemannian structure that it inherits from $(\cNo, W_2)$.

The barycenter problem seeks to generalize the notion of averages to non-Euclidean spaces. In the 2-Wasserstein barycenter problem, one seeks a solution to
\begin{equation}\label{eq:bary}
	\min_{b \in \cP_2(\R^d)} \frac{1}{2}\E_{\mu \sim Q} W_2^2(\mu, b),
\end{equation}
where $Q$ is a distribution over $\cP_2(\R^d)$ with finite second moment, which we write as $Q \in \cP_2(\cP_2(\R^d))$. When $Q$ is supported on Gaussians, the minimum is achieved on Gaussians \citep{knott1994generalization,agueh2011barycenter,alvarez2016fixed}. For $\cN_0(\R^d)$, due to the identification in \eqref{eq:GaussWass}, this is equivalent to the Fr\'echet mean of PSD matrices on the \BW\ manifold. Without loss of generality, we think of $Q$ as a distribution over PSD matrices.

We finally arrive at the connection between low-rank PSD matrix recovery and \BW\ barycenters. The following proposition connects the barycenter problem~\eqref{eq:bary} when $Q = Q_n = \frac{1}{n} \sum_{i=1}^n \delta_{N(\bzero, y_i \bx_i \bx_i^\top)}$ to our new low-rank matrix recovery program \eqref{eq:lsmatrec_new}, provided that $\frac{1}{n} \sum_{i=1}^n \bx_i \bx_i^\top = \bI$. This proposition indicates that we can recover the matrix $\bS$ by solving a Wasserstein barycenter problem.

\begin{pro} \label{prop:equiv}
If $\frac{1}{n} \sum_i \bx_i \bx_i^\top = \bI$, then
	\begin{align}\label{eq:bwmr}
	    \argmin_{\bSigma \in \bbS_+^d} &\frac{1}{2n} \| \sqrt{\cA(\bSigma)} - \sqrt{\by}\|_2^2 = \argmin_{\bSigma \in \bbS_+^d} \frac{1}{2n} \sum_{i=1}^n \dbw^2(\bSigma, y_i \bx_i \bx_i^\top).
	\end{align}
\end{pro}

To make this result practical, we cannot assume in general that $\frac{1}{n} \sum_i \bx_i \bx_i = \bI$. If we instead encounter a case where $\frac{1}{n} \sum_{i=1}^n \bx_i\bx_i^\top = \bC_n$, where $\bC_n$ is the PD sample covariance matrix of the vectors $\bx_1, \dots, \bx_n$, then the transformation $\bx_i \mapsto \bC_n^{-1/2} \bx_i$ outputs vectors with identity covariance (i.e., $\frac{1}{n} \sum_{i=1}^n\bC_n^{-1/2} \bx_i\bx_i^\top\bC_n^{-1/2} = \bI$). We are able use this fact to recover the matrix $\bS$, as we show in the following proposition.

\begin{pro}\label{prop:whitenedbary}
	 Let $\bC_n = \frac{1}{n} \sum_{i=1}^n \bx_i \bx_i^\top \in \bbS_{++}^d$. Then 
	 \begin{align}
	     \argmin_{\bSigma \in \bbS_{+}^d} &\frac{1}{2n}\| \sqrt{\cA(\bC_n^{-1/2} \bSigma \bC_n^{-1/2})} - \sqrt{\by}\|^{2} = \argmin_{\bSigma \in \bbS_{+}^d} \frac{1}{2n} \sum_{i=1}^n \dbw^2(\bSigma, y_i \bC_n^{-1/2} \bx_i \bx_i^\top  \bC_n^{-1/2}),
	 \end{align}
	 and $\bC_n^{1/2} \bS \bC_n^{1/2}$ is a solution to both problems.
\end{pro}

Notice that one can recover $\bS$ from $\bC_n^{1/2} \bS \bC_n^{1/2}$ as $\bS = \bC_n^{-1/2} \bC_n^{1/2} \bS \bC_n^{1/2} \bC_n^{-1/2}$.
Thus, in the sample setting where $\bC_n$ is not exactly the identity and assuming we can solve the barycenter problem, we envision a two stage procedure: 1) recover the barycenter $\bSigma_n$ of the $n$ matrices $y_1 \bC_n^{-1/2} \bx_1 \bx_1^\top \bC_n^{-1/2}, \dots, y_n \bC_n^{-1/2} \bx_n \bx_n^\top \bC_n^{-1/2}$, and 2) transform the barycenter by $\bC_n^{-1/2} \bSigma_n \bC_n^{-1/2}$ to find $\bS$. 

The whitening step can be efficiently computed since we can use any linear transformation $\bL^{-1}$ such that $\bL^{-1} \frac{1}{n} \sum_{i=1}^n  \bx_i \bx_i^\top \bL^{-1\top} = \bI$. For example, with the Cholesky factorization $\bC_n = \bL \bL^\top$, we can solve the equations $\bL \bz_i = \bx_i$ for $\bz_i$, and these satisfy $\frac{1}{n} \sum_i \bz_i \bz_i^\top = \bI$.

The connections established by Propositions \ref{prop:equiv} and \ref{prop:whitenedbary} enables the development of novel methods for the matrix recovery problem, since we can solve a specific Wasserstein barycenter problem rather than the original matrix recovery problem \eqref{eq:lsmatrec_new}. In other words, any methods that solve this Wasserstein barycenter problem could be used to solve the matrix recovery problem. Since barycenters are geometric notions of averages, this naturally leads to the development of novel geometric methods for matrix recovery.

\subsection{Algorithms for Barycenters}
\label{subsec:algos}

The primary way to compute \BW\ barycenters involves Riemannian gradient descent~\citep{alvarez2016fixed,chewi2020gradient,altschuler2021averaging}. 
Following the results in the last section, we wish to find the \BW\ barycenter of the matrices $\bX_i = y_i \bC_n^{-1/2}\bx_i \bx_i^\top \bC_n^{-1/2}$, $i=1, \dots, n$. In other words, we seek to minimize the energy function $F:\bbS_+^d \to \R$ given by
\begin{align}\label{eq:fcnlemp}
F(\bSigma) &= \frac{1}{2n} \sum_{i=1}^n \dbw^2(\bX_i, \bSigma).
\end{align}
For ease of notation, we will write $F$ as an expectation over $y_i \bC_n^{-1/2}\bx_i \bx_i^\top \bC_n^{-1/2} \sim Q$, and whether or not $Q$ is a discrete distribution will be made clear from context.

The gradient of $F$ at full rank $\bSigma_0$ is
\begin{align*}
\nabla F(\bSigma_0) &= \bI - \E_{\bX \sim Q} \frac{\bX}{\sqrt{\Tr(\bX \bSigma_0)}} =: \bI - \tilde{T}(\bSigma_0) ,
\end{align*}
where for convenience we have defined the quantity $\tilde{T}(\bSigma_0)$. At low-rank $\bSigma_0$, this is a subgradient \citep{clarke1990optimization}.
We note that this holds for general measures $Q$ where we can differentiate under the integral.

Wasserstein gradient descent uses the gradient to determine a geodesic along which to move. In Wasserstein space, this is the ``pushforward" direction in a base measure is transported. For a complete description of Wasserstein gradient descent, see \citep{alvarez2016fixed,zemel2019procrustes,chewi2020gradient,altschuler2021averaging}, and for a more complete discussion of \BW\ geometry, the reader should consult \cite{bhatia2019bures}. We have included a discussion of the geodesic structure of \BW\ space in the appendix.  For our purposes, we consider \BW\ gradient descent (\BW GD) with step size $\eta_k$:
\begin{equation}\label{eq:gd}
	\bSigma_{k+1} = (\bI - \eta_k\nabla F(\bSigma_k)) \bSigma_k (\bI - \eta_k \nabla F(\bSigma_k)).
\end{equation}
When $\eta_k = 1$, this corresponds to the fixed point iteration of \cite{alvarez2016fixed}.
Note that this is easy to extend to the stochastic setting: if we observe a stochastic gradient $\bG_k$ rather than $\nabla F(\bSigma_k)$, then the \BW\ stochastic gradient descent (\BW SGD) iteration would use $\bG_k$ in place of $\nabla F(\bSigma_k)$.
For example, one common variant of such a stochastic gradient method uses
\begin{equation}\label{eq:gradss}
	\bG_k = \bI - \frac{\bX_k}{\Tr(\bX_k \bSigma_k)^{1/2}},
\end{equation}
where $k=1, \dots, n$. In other words, this variant of stochastic gradient descent passes over each sample one at a time. At each point in time, we take a gradient with respect to that sample alone and move in the minus gradient direction.

To save computational time and to allow efficient computation in the low-rank case, we modify the \BW\ gradient descent iteration \eqref{eq:gd} and the \BW SGD iteration to instead operate on factorized matrices. This means that instead of storing the sequence $\bSigma_k$ for $k \in \N$, we instead store the sequence
\begin{equation}\label{eq:frgd}
	\bU_{k+1} = (\bI - \eta_k \nabla F(\bU_k \bU_k^\top)) \bU_k.
\end{equation}
We note that if $\bSigma_k$ is low-rank in \eqref{eq:gd}, then the update in \eqref{eq:frgd} is equivalent to \eqref{eq:gd}. In particular, it is not hard to show that if $\bU_k \bU_k^\top = \bSigma_k$, then $\bU_{k+1} \bU_{k+1}^\top = \bSigma_{k+1}$. 
In the same way as before, stochastic gradient methods naturally extend to the low-rank setting. This update corresponds to a geodesic gradient descent update over the fixed rank \BW\ manifold \citep{massart2020quotient}.

When $Q = Q_n$, the \BW\ gradient descent updates can be computed in $O(ndr)$ time. This follows from the fact that we can rewrite the update in \eqref{eq:frgd} as
\begin{equation}
    \bU_{k+1}	= (1-\eta_k) \bU_k + \frac{\eta_k}{n} \sum_{i=1}^n \frac{\bx_i \bx_i^\top \bU_k}{\|\bU_k^\top \bx_i\|}.
\end{equation}
In the case of single-sample streaming \BW SGD, which uses the gradient in \eqref{eq:gradss}, the updates take $O(dr)$ time.
We also note that both the \BW\ gradient descent and \BW SGD iterations maintain the rank of the updated matrix for all $\eta_k \in [0, 1)$. For $\eta_k = 1$, the rank is maintained for \BW\ gradient descent as long as $\rank(\sum_i \bx_i \bx_i^\top) = d$. 

\subsection{Regularization through Perturbed Gradient Descent}
\label{subsec:regalgo}

As we demonstrate later, our current theory only works for $r \geq 3$. To extend our results to the case of $r=1$ or $r=2$, we develop a regularized method. 
Consider the observation model $y_i = \langle \bx_i \bx_i^\top, \bS\rangle$ for a rank $r$ matrix $\bS$ and $i=1, \dots, n$. If we choose an arbitrary rank $r'$ matrix $\bDelta$, we could instead try to recover the rank at most $r+r'$ matrix $\bS + \bDelta$ from the observations $\tilde y_i = \langle \bx_i \bx_i^\top, \bS + \bDelta \rangle = \langle \bx_i \bx_i^\top, \bS\rangle + \langle \bx_i \bx_i^\top,  \bDelta \rangle$, which can be computed by computing and adding the factors $\langle \bx_i \bx_i^\top,  \bDelta \rangle$ to the observations $y_i$. If we then find the barycenter of $\tilde y_i \bx_i \bx_i^\top$ (provided that it is unique, which we prove later), then this would recover $\bS + \bDelta$! Since this is true for any such $\bDelta$, it can be picked by the user beforehand. One can then recover $\bS$ from simple subtraction: $\bS =\bS + \bDelta - \bDelta$. In brief, every rank $r$ recovery problem can be solved by instead first solving the rank $r + r'$ problem to find $\bS + \bDelta$ and then subtracting off the perturbation factor $\bDelta$.

\section{Theoretical Results}
\label{sec:theory}

We now discuss the main theoretical results of this paper. 
We make the following assumption on our measurement model in \eqref{eq:matrecmeas}.
\begin{assump}\label{assump:specgauss}
	We observe data from the model \eqref{eq:matrecmeas} with $\bx_i \overset{i.i.d.}{\sim} N(\bzero, \bI)$. 
	The underlying matrix $\bS$ is rank $r$ and satisfies $m \leq \lambda_r(\bS)  \leq \lambda_1(\bS) \leq M$. 
\end{assump}

We believe that the assumption of Gaussianity can be weakened to sub-Gaussianity without too much trouble. The sub-Gaussianity is essential for an $\ell_2/\ell_1$ restricted isometry property~\cite{chen2015exact,cai2015rop} that we need to hold (see Appendix \ref{subsec:ripcond}).
Our main result on matrix recovery with gradient descent for \BW\ barycenters is given in the following theorem.

\begin{thm}\label{thm:main}
	Suppose that we observe $y_i=\langle \bx_i \bx_i^\top, \bS\rangle$, $i=1, \dots, n$, where $\bx_i$ and $S$ satisfy Assumption~\ref{assump:specgauss}. Suppose further that $r = \rank(\bS) \geq 3$, and let $\bS_n = \bC_n^{1/2} \bS \bC_n^{1/2}$. Then, for constants $c_1, c_2$, if $n \gtrsim dr$, with probability at least $1-\exp(-c_2 n)$, 
	\begin{enumerate}
		\item $\bS_n$ is the unique global minimizer of $F$ over $\bbS_+^d$.
		\item Let $\bU_0$ be the initial iterate of \BW\ gradient descent. If $\lambda_1(\bU_0\bU_0^\top) \leq M$, and $F(\bU_0\bU_0^\top) - F(\bS_n) \leq  \frac{c_1^{5}m^8r^2 \beta^{10}}{6^5 M^{15/2}d^5}$ for an $\bS$ dependent constant $\beta$, then \BW\ gradient descent with step size $1$ satisfies the following bound for $\cC = O\big(c_1 m / M^{3/2}\big)$.
	\end{enumerate}
	\vspace{-.5cm}
	\begin{align}
	\dbw^2 (\bU_k \bU_k^\top,& \bC_n^{1/2} \bS \bC_n^{1/2}) \leq (1- \cC)^k (F(\bU_0 \bU_0^\top) -F(\bS) ).
	\end{align}
\end{thm}
This is the first result for convergence in Bures-Wassersten distance in the literature. The constant $\cC$ is the local strong geodesic convexity constant seen in \eqref{eq:locsc}. 
We note that this amounts to showing that our new energy given by \eqref{eq:lsmatrec_new} has a positive definite Hessian in a neighborhood around $\bS$. A couple of remarks are in order to discuss two issues that arise with the analysis of our method: initialization and the rank constraint on $\bS$. In practice, we observe convergence to $\bS$ from random intialization, but do not currently have a proof of this fact.

\begin{rem}
	We note that the convergence bound in Theorem \ref{thm:main} is local. A few procedures can be used to initialize in the correct neighborhood. First, under the Gaussian assumption, $\E y (\bx\bx^\top-\bI) = 2 \bS$, and so one could use a rank $r$ approximation of $\frac{1}{2n} \sum_{i=1}^n y_i (\bx_i \bx_i^\top - \bI)$ to initialize the gradient descent. This approximation can be computed in $O(ndr)$ time with the power method. A finite sample approximation result follows from concentration of the fourth moment tensor, see \cite[Theorem 4.13]{diakonikolas2019robust} for details. 
	Another path to initialization would be to use gradient descent on full rank PSD matrices in the first stage since, as we show in the Appendix, the function \eqref{eq:bwmr} is convex over PD matrices. This would allow one to recover a good approximation to $\bS$, and take a rank $r$ approximation of it, and then run \BW\ gradient descent from there.
	The downside of this method is that it takes complexity $O(nd^2)$ to compute, but it would not need concentration of the fourth moment tensor. Also, the gradient descent procedure with overspecified rank converges sublinearly. 
\end{rem}

\begin{rem}
	The result holds for $\rank(\bS) \geq 3$. This is due to a smoothness bound that requires an expectation of the form $\E \sqrt{1 + (x_{d+1}^2 + \dots + x_d^2)/(x_1^2 + \dots + x_r)^2}$ for $x_i \overset{i.i.d.}{\sim} N(0,1)$ random variables. In order for $\E 1/(x_1^2 + \dots + x_r)^2$ to exist, we need $r \geq 3$. In practice, and as we show in the experiments, the method still succeeds much of the time for $r=1, 2$. To have a theoretically guaranteed method in these settings, one can use the regularized \BW\ gradient descent method of Section \ref{subsec:regalgo}. 
\end{rem}

We give a sketch of the proof of Theorem \ref{thm:main}.
\begin{itemize}
	\item We first show that the energy $F$ in \eqref{eq:fcnlemp} is Euclidean strongly convex over $\bbS_+^d$ and that $\bS_n$ is the unique minimizer with high probability. This result uses the $\ell_2/\ell_1$-RIP condition of \cite{chen2015exact} (or restricted uniform boundedness condition of \cite{cai2015rop}).
	\item We then prove a smoothness result: $\|\nabla F(\bSigma)\| \lesssim \|\bSigma - \bS\|_F^{1/2}$ for rank $r$ matrices $\bSigma$.
	\item Using this smoothness result, we are able so show that $F$ is locally-geodesically convex (with respect to \BW\ geodesics) around $\bS_n$ in $\bbS_+^{d, r}$.
	\item Local strong convexity along with a geodesic smoothness result yields the local linear convergence result.
\end{itemize}


We also include the following theorem on the convergence of \BW SGD. This guarantees a slow rate of convergence for \BW SGD with gradient given by \eqref{eq:gradss}.
\begin{thm}\label{thm:BWSGD}
	Suppose that we observe $y_i=\langle \bx_i \bx_i^\top, \bS\rangle$, $i=1, \dots, n$, where $\bx_i$ and $S$ satisfy Assumption~\ref{assump:specgauss}. Suppose that we run single sample streaming \BW SGD, which uses gradient \eqref{eq:gradss}, for $n$ iterations with step size $1/\sqrt{n}$. Then, for a constant $c_1$, if $n \gtrsim dr$, with probability at least $1-\exp(-c_1n)$,
\begin{equation}
	\min_{k=1, \dots, n} \E  \| \nabla F(\bSigma_k)\|^2_{\bSigma_k}  = O\big(n^{-1/2}\big),
	\end{equation}
	where $\|\cdot\|_{\bSigma}^2 = \E_{\bz \sim N(\bzero, \bSigma)} \|\cdot \bz\|_2^2$ is the norm induced by the \BW\ Riemannian metric.
\end{thm}

This states that the best iterate of the \BW SGD sequence outputs an approximate stationary point with respect to the norm induced by the \BW\ Riemannian metric. However, we cannot guarantee that this stationary point is the global minimum. 
We comment that we do not tend to run into spurious local minima in practice, and future work should go into studying this fact. While we do not have a justification for this, we give a theorem on the $r=1$ case, where we show that the energy function has no local minima in the asymptotic limit.  The proof of this theorem is left to the supplementary material.
\begin{thm}
	Consider the observation model with the rank one matrix $\bS = \bv\bv^\top$  and $y = \langle \bx\bx^\top, \bS \rangle$. Then, the only fixed points of the population version of \eqref{eq:frgd} (which corresponds to gradient descent on \eqref{eq:fcnlemp} with the sum replaced by an integral)
	are $\bv$ or orthogonal to $\bv$. In particular, this implies that population \BW\ gradient descent  from any initialization such that $\bu_0 \not \perp \bv$ converges to $\bv$.
\end{thm}
For the case of general $r$, we believe that a similar result holds, although we have not yet been able to show it. Furthermore, we also believe that these results can be extended to high probability results in the finite sample case.

\begin{figure}[h!]
    \centering
    \vspace{-.0cm}
    \includegraphics[width = 0.8 \columnwidth]{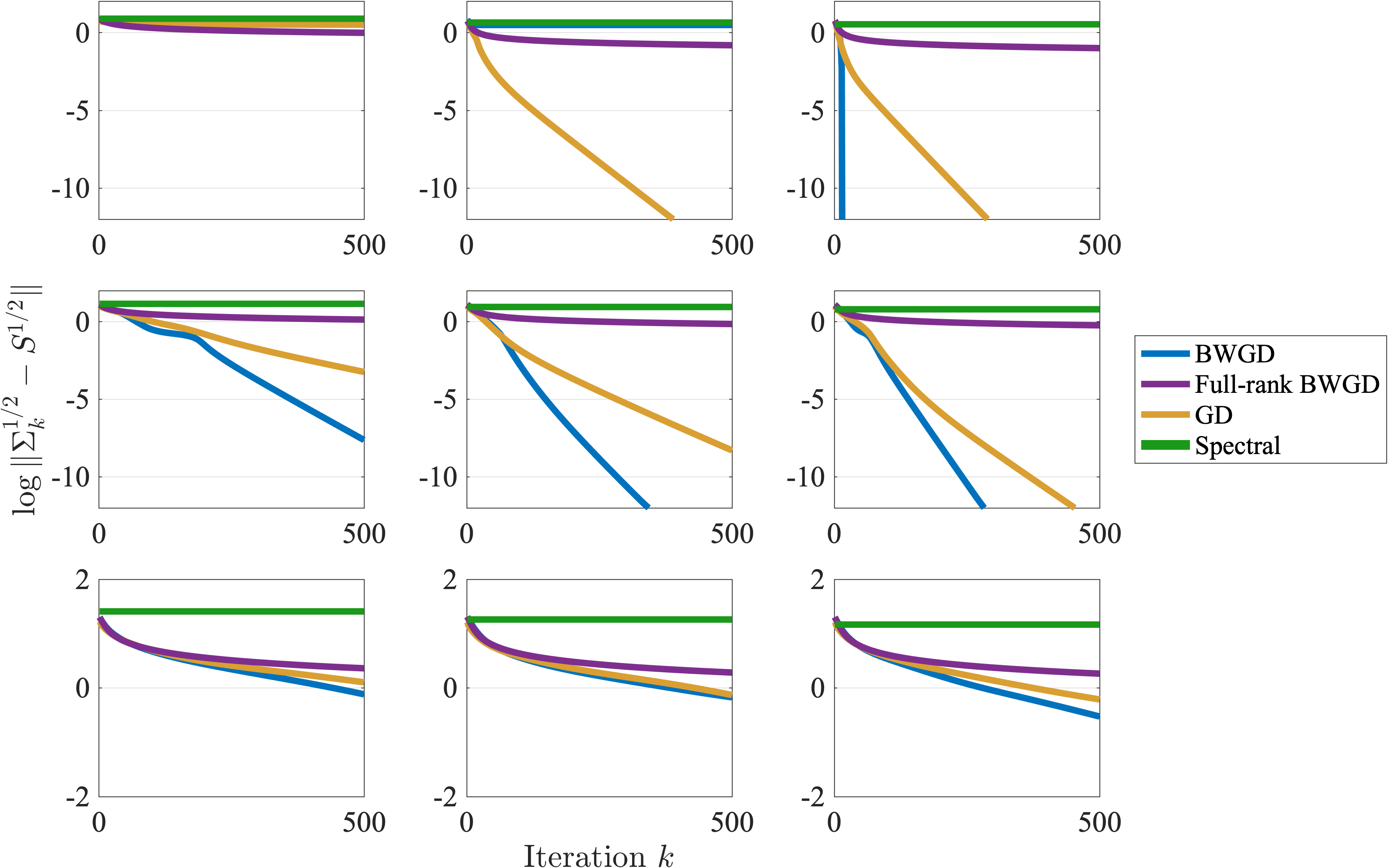}
    \caption{Convergence of \BW\ gradient descent \eqref{eq:frgd} with the full gradient and GD \citep{li2019nonconvex}.Here, $d=32$, and down the columns we use $r=1, 4, 16$, respectively. Across the rows we vary the number of points, with $n=3dr$, $n=10dr$, and $n=20dr$, respectively. As we can see, for low to moderate ranks, the \BW\ gradient descent method converges much faster than the standard GD method.}
    \label{fig:conv1}
\end{figure}

\begin{figure}[h!]
    \centering
    \vspace{-.0cm}
    \includegraphics[width = 0.8 \columnwidth]{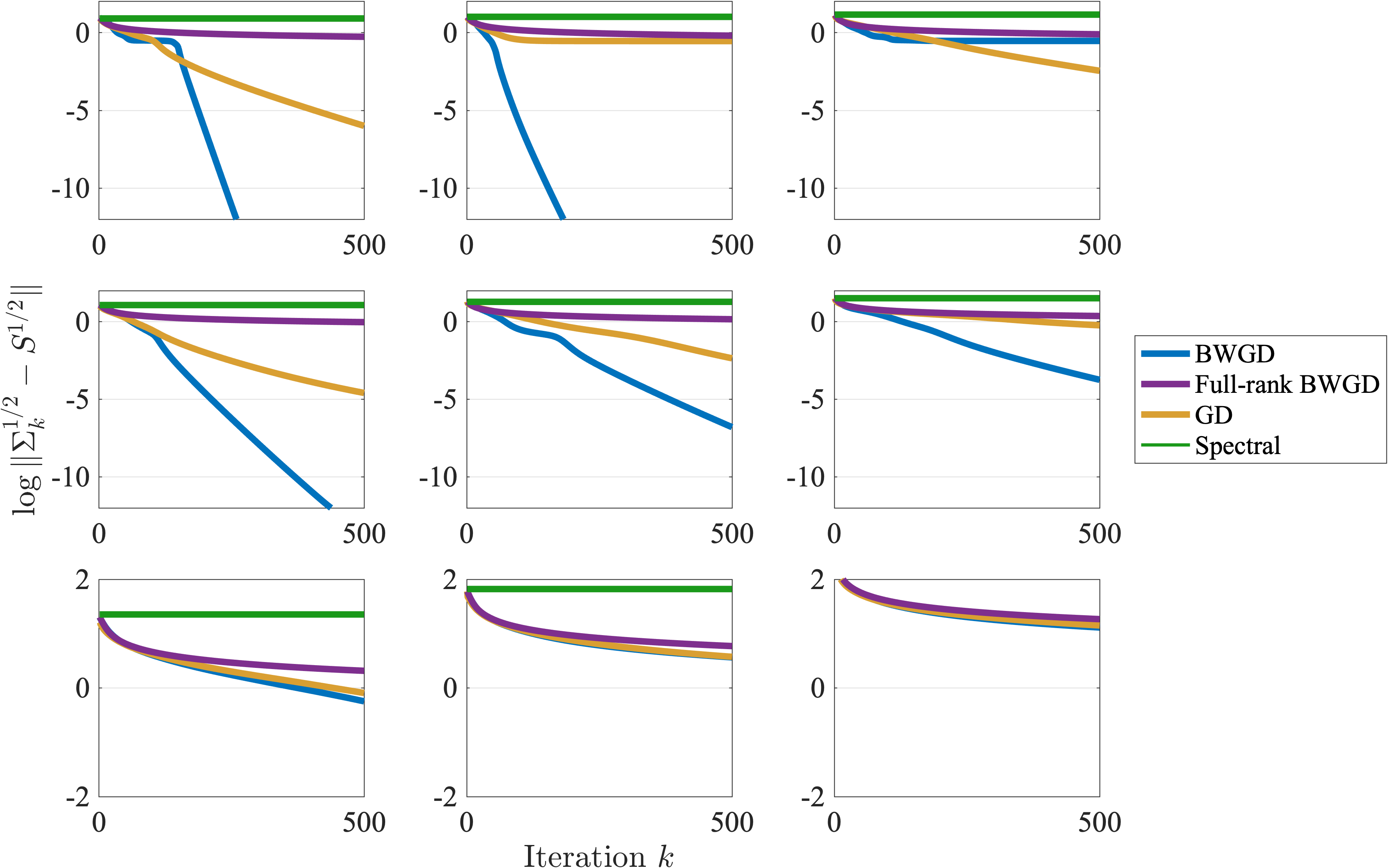}
    \caption{Convergence of \BW\ gradient descent \eqref{eq:frgd} with the full gradient and GD \citep{li2019nonconvex}. Here, $d=32$, and down the columns we use $r=2, 4, 16$, respectively. Across the rows we vary scale factor $\alpha$, with $\alpha = 0, 1, 2$, respectively. As we can see, the \BW\ gradient descent method converges much faster than the standard GD method.}
    \label{fig:conv2}
\end{figure}



\section{Experiments}
\label{sec:exp}

Here we present some numerical simulations that demonstrate the effectiveness of our method. All experiments were run on a 2020 Macbook Pro with a quad core CPU and 16 GB of RAM.

\subsection{Synthetic Experiments}

We begin with some experiments on generated datasets to better understand the performance of our method. Since our main results focus on the performance of gradient descent rather than stochastic gradient descent, we focus on its performance here.

The methods we compare are rank $r$ \BW\ gradient descent, full rank \BW\ gradient descent, the nonconvex Euclidean Gradient Descent (GD) of \cite{li2019nonconvex}, and a spectral method, which takes a low-rank approximation to $\bS_n = \frac{1}{2n}\sum_{i=1}^ny_i(\bx_i \bx_i^\top - \bI)$ since $\frac{1}{2}\E y^2 (\bx\bx^\top - \bI) = \bS$~\citep{sedghi2016provable}. As an error metric, we compute $\|\bSigma_k^{1/2} - \bS^{1/2}\|_F= \Theta(\dbw(\bSigma_k, \bS))$. We do not compare with other matrix recovery algorithms since they are not guaranteed to work in the symmetric rank one projection setting. Also, we find the comparison with GD to be the most relevant, since \BW\ gradient descent and GD have comparable complexity and are both first-order methods. 

In our first experiment, we test the accuracy of the methods over time as we vary the rank of $\bS$ and the sample size. Since the spectral method is not iterative, we include it as a horizontal line. 
Figure \ref{fig:conv1} displays the results of this experiment. Across the rows we vary the number of points and down the columns we vary the rank of $\bS$. The fixed dimension is $d=32$, the ranks from top to bottom are $r=1, 4, 16$, and across the rows the number number points are $3dr$, $10dr$, and $20dr$, respectively. For each frame, we generate 20 datasets and run the four methods on them.
All methods are run with random initialization, where the entries of $\bU_0$ are i.i.d.~$N(0,1)$.
For $r=1$, we see that rank $r$ \BW\ gradient descent succeeds once the number of points is sufficiently large. Furthermore, the convergence when it is successful is extremely fast. For moderate ranks, rank $r$ \BW\ gradient descent converges faster than the previous GD method of \cite{li2019nonconvex}. 

In Figure \ref{fig:conv2}, we examine the performance of the methods under varying conditioning of $\bS$. We set $d=32$ and $n=5dr$ and vary $r$ as well as the conditioning of the matrix $S$. Here, $\bS = \bV \diag(r^\alpha, (r-1)^\alpha, \dots, 1^\alpha) \bV^\top$, where the entries of $\bV$ are i.i.d. $N(0,1)$ and $\alpha$ is a scale factor. Figure \ref{fig:conv2} displays the results on 20 randomly generated datasets per frame. The rows correspond to $r=2, 4$, and $16$ respectively. The columns correspond to $\alpha = 0, 1$ and $2$ respectively. Rank $r$ \BW\ gradient descent performs uniformly well throughout.

In Figure \ref{fig:samp} we show the dependence on sample size. Here, $d = 64$ and $r$ is varied from $1$ to $20$. The error of \BW\ gradient descent and GD after 200 iterations is shown for sample sizes of $d$, $2d$, $\dots$, $20d$ for each value of $r$. For each $r, n$ pair, we generate 20 datasets and compute the average error across them. The color indicates the average error value across these datasets. As we see, \BW\ gradient descent performs the best out of these methods. GD with linesearch is also competitive, but is more time consuming.

For the final synthetic experiment in Figure \ref{fig:highdim}, we demonstrate the scalability of \BW\ gradient descent to higher dimensions. We note that it scales much better in terms of actual computational time when compared with the Euclidean GD method of \cite{li2019nonconvex}.

\begin{figure}[h!]
    \centering
    \vspace{-.0cm}
    \includegraphics[width = 0.8 \columnwidth]{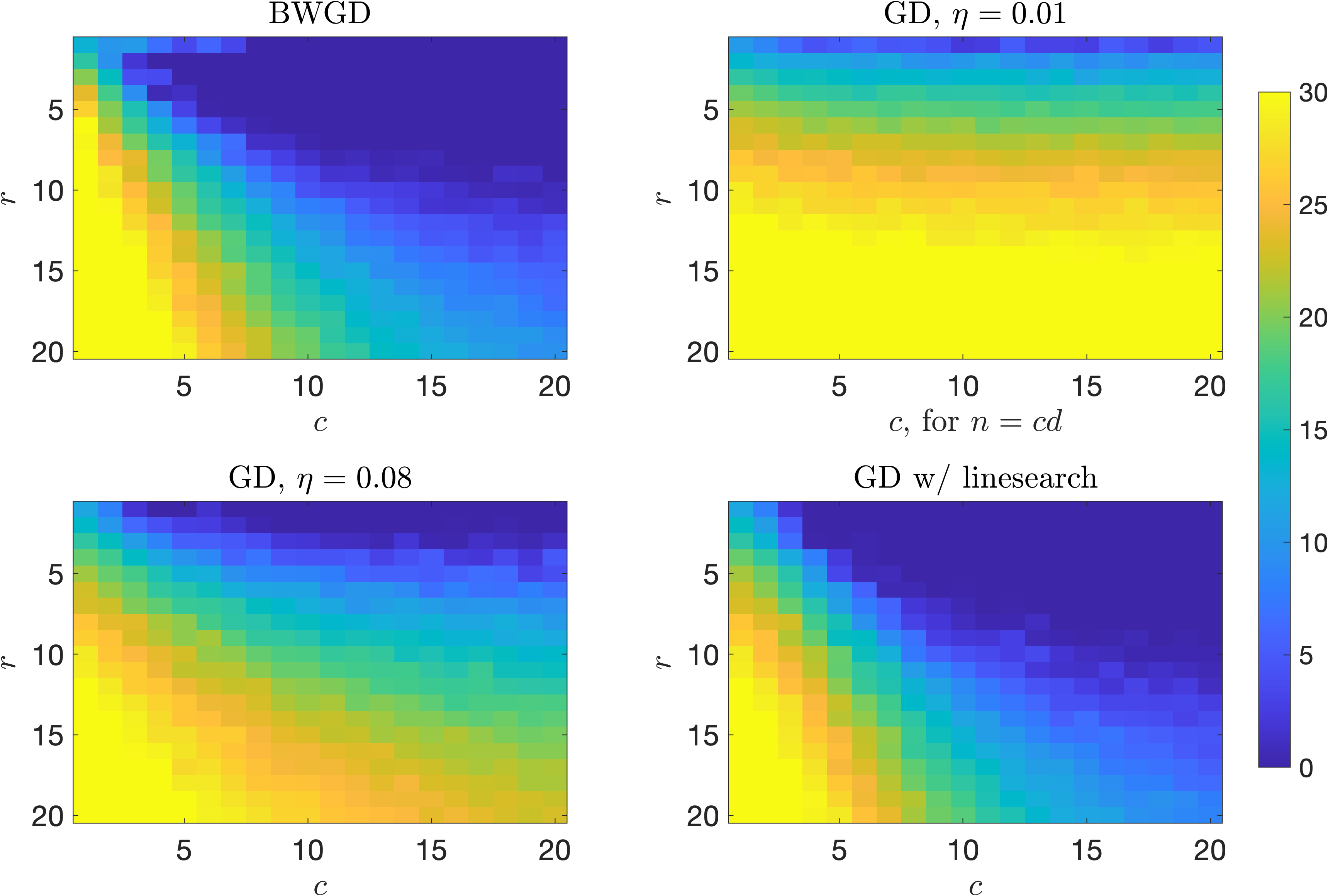}
    \caption{Convergence of \BW\ gradient descent \eqref{eq:frgd} with the full gradient and GD \citep{li2019nonconvex}. Here, $d=64$, and down the columns of each inset we use $r=1, \dots, 20$, respectively. Across the rows we vary the scale factor $c$, where $n = cd$. As we can see, the \BW\ gradient descent method converges much faster than the standard GD method with fixed step size, and performs on par with EGD with linesearch.}
    \label{fig:samp}
\end{figure}

\begin{figure}[h!]
    \centering
    \vspace{-.0cm}
    \includegraphics[width = 0.5 \columnwidth]{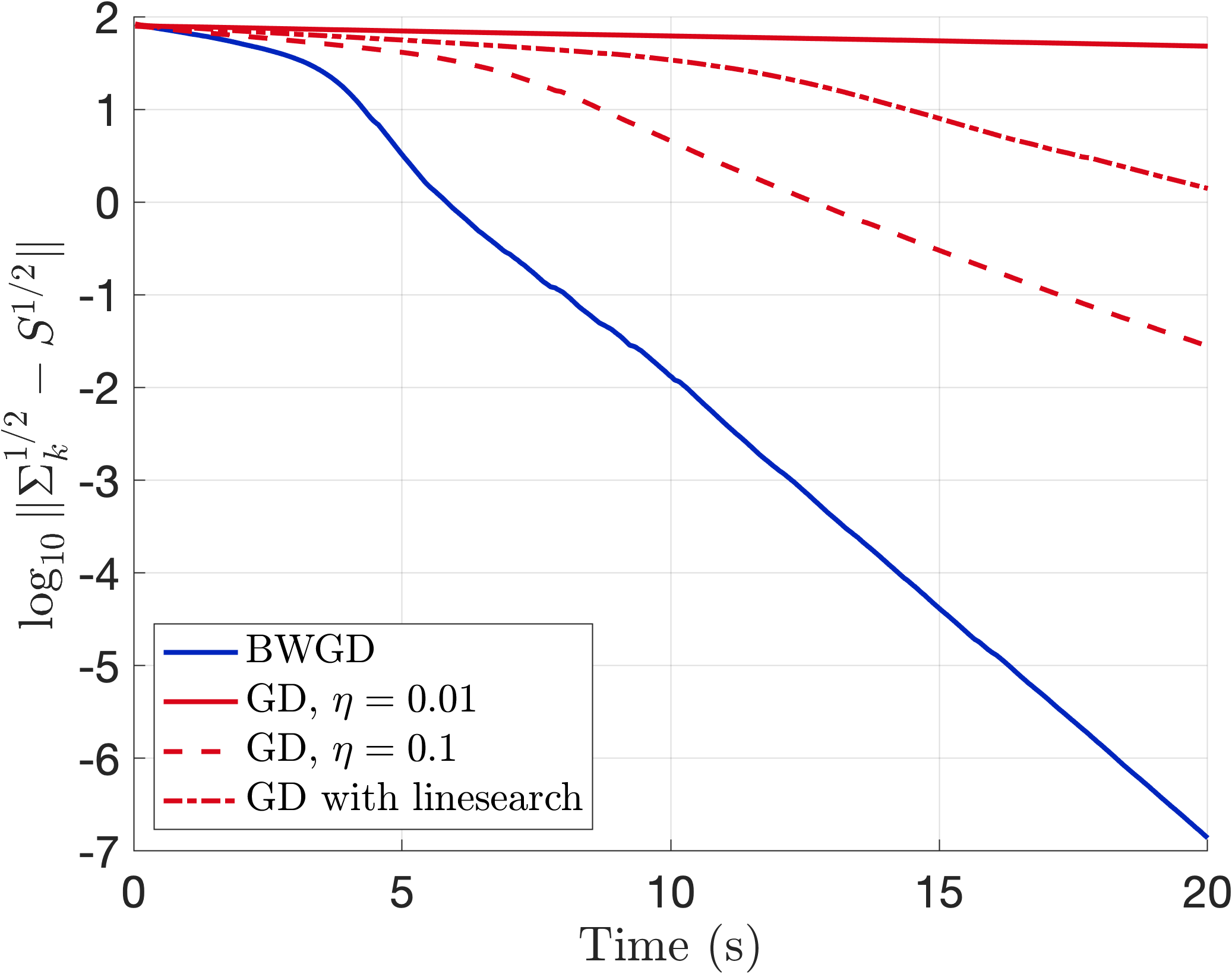}
    \caption{Convergence of \BW\ gradient descent \eqref{eq:frgd} and GD \citep{li2019nonconvex}. Here, $d=512$, $r=8$, and $n=3rd$. As we can see, the \BW\ gradient descent method converges much faster than the standard GD method with any choice of step sizes. The per-iteration convergence rate for \BW\ gradient descent is comparable to GD with linesearch, but doing linesearch increases computational burden.}
    \label{fig:highdim}
\end{figure}


\subsection{Real Data Experiment: Phase Retrieval}

Here we give the details of the experiment displayed in Figure~\ref{fig:pr}.
We replicate the phase retrieval experiments in \cite{candes2015phase}. In this paper, the authors study the nonconvex Wirtinger Flow algorithm. Since phase retrieval is equivalent to the recovery problem in \eqref{eq:matrecmeas} with a rank one complex $\bS$, we can apply our algorithm in this setting.

Here, we use $i = \sqrt{-1}$.
We use an image of Denali National Park, which is denoted by the 2-dimensional array $\bJ$ for each color band. In our simulated acquisition model, for $m = (u, v, \ell)$, we acquire data of the form
\begin{align}
    y_m &= \Big|  \sum_{j=1,k=1}^{j=d_1, k=d_2} \bJ_{jk} \bar{d}_\ell(j,k) e^{-i 2\pi (ju + kv)} \Big|^2.
\end{align}
Here, $d_\ell(j,k)\sim b_1b_2$, where $b_1$ is uniform on $\{1, -1, i, -i\}$, and $b_2$ takes values $\sqrt{2}/2$ with probability $4/5$ and $\sqrt{3}$ with probability $1/5$. The goal is to recover the image $\bJ$ from these measurements. We note that these measurements can be equivalently written as $y_m = \bF_{u,v,\ell}^{\flat \top} \bJ^\flat \bJ^{\flat \top} \bF_{u,v,\ell}^\flat,$
where $\cdot^\flat$ denotes the vectorization operation and $\bF$ is a matrix with entries $\bar{d}_\ell(j,k) e^{-i 2\pi (ju + kv)}$. This notation makes clear the connection with the original matrix recovery problem in \eqref{eq:matrecmeas}. We display the errors versus runtimes in Figure \ref{fig:pr2}. Despite the fact that this example has $r=1$, \BW\ gradient descent still efficiently recovers the underlying image, even though our current theory only works for $r \geq 3$. 

\begin{figure}[h!]
    \centering
    \includegraphics[width = .49\columnwidth]{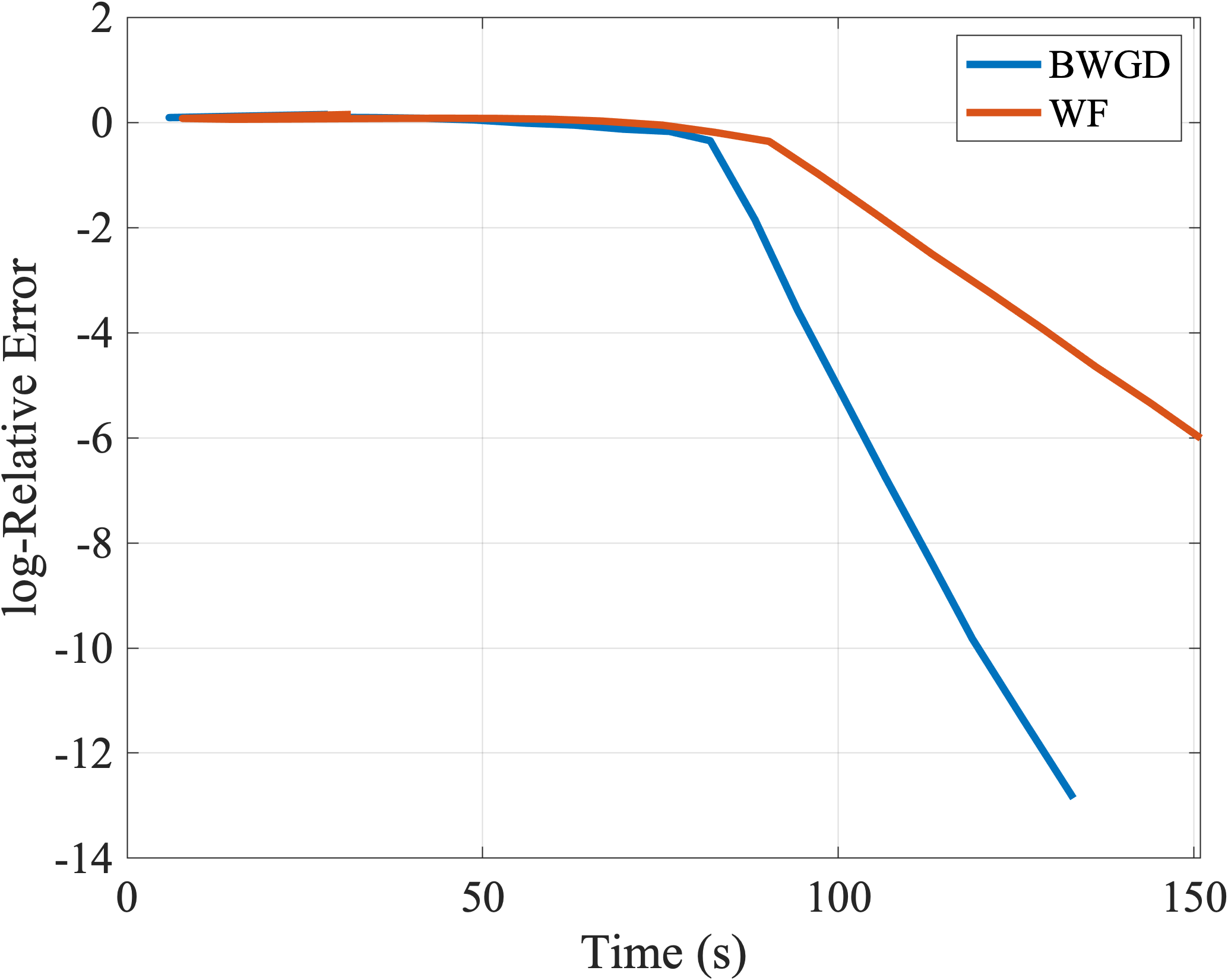}
    \caption{Error vs Runtime of \BW\ gradient descent (our method) and Wirtinger Flow~\cite{candes2015phase} on the phase retrieval experiment in Figure \ref{fig:pr}. As we can see, \BW\ gradient descent recovers the underlying image faster. Both methods are initialized with the power method as is done in \citep{candes2015phase}.}
    \label{fig:pr2}
\end{figure}

\section{Limitations}
\label{sec:lim}

There are a few notable limitations for the current work. First of all, the theory does not directly extend to the cases of $r=1$ and $r=2$, and we must resort instead to the regularized methods discussed in Section \ref{subsec:regalgo}. It is unclear whether or not this is a limitation of the methods or the analysis, although experiments indicate that the method works well for small dimensions in practice. Second, while our experiments indicate well-behaved energy landscapes, we are only able to prove local convergence results. Third, while we give the first optimization rates in terms of \BW\ distance for this problem, our constants are not optimized. Fourth, we assume that one knows the rank of $S$ in advance. This is not a large issue since our framework allows for one to pick any $r' \geq r$ and still recover $r$, albeit at a slower rate.

\section{Conclusion}

We have presented a novel connection between the \BW\ barycenter problem and low-rank matrix recovery from rank one measurements. We show that a novel energy minimization problem coincides with the barycenter minimization. This connection allows us to extend algorithms from the barycenter problem to the matrix recovery problem, giving new algorithms for recovering low-rank PSD matrices. Our methods are guaranteed to have local linear convergence in \BW\ distance, which is a stronger guarantee than existing methods.

This work leaves open many unexplored directions. For example, it would be interesting to show that the energy landscape does not exhibit spurious local minima. Beyond this, it would also be interesting to know when and how one can go beyond the RIP assumption, which currently relies on sub-Gaussianity of the sensing vectors. Finally, it would be interesting to connect these ideas to optimization landscapes for neural networks \citep{zhong2017recovery,du2019gradient,li2019nonconvex}.

\section{Acknowledgements}
PR is supported by NSF grants IIS-1838071, DMS-2022448, and CCF-2106377.

\bibliographystyle{plainnat}
\bibliography{refs}

\appendix

\section{Geometries of $\bbS_+^d$}

Since the \BW\ barycenter problem is inherently a geometric minimization problem over $\bbS_+^d$, we will quickly comment on how the choice of geometry effects optimization algorithms over this space.

There are many ways to define geometries over $\bbS_+^d$. For sake of comparison with previous methods for matrix recovery, we will compare \BW\ space to the Euclidean geometry over PSD matrices. Among other things, the choice of geometry gives us a way of defining defining distance minimizing paths, or \emph{geodesics}, over $\bbS_+^d$. 

Consider two matrices $\bSigma_0, \bSigma_1 \in \bbS_+^d$ such that $\rank(\bSigma_0 \bSigma_1) = \rank(\bSigma_0) =  \rank(\bSigma_1)$. This is a sufficient but not necessary condition to ensure that our following definition of \BW\ geodesic is well defined since the same map can work for transporting higher rank PSD matrices to lower rank matrices.
In any case, under this assumption, there exists a transport map from $\bSigma_0$ to $\bSigma_1$ given by
\begin{equation}
    \bT = \bSigma_1^{1/2} (\bSigma_1^{1/2} \bSigma_0 \bSigma_1^{1/2})^{-1/2} \bSigma_1^{1/2},
\end{equation}
where the inverse is actually a pseudoinverse and one can check that $T \bSigma_0 T = \bSigma_1$.
 In this case, the Euclidean and \BW\ geodesics $\bSigma_t:[0,1] \to \bbS_+^d$ are given by
\begin{align}
	\bSigma_t &= (1-t) \bSigma_0 + t \bSigma_1, \tag{\sf{EG}} \label{eq:eg}\\
	\bSigma_t &=(\bI + t(\bT - \bI))\bSigma_0(I + t(\bT - \bI)).\tag{\sf{BWG}} \label{eq:bwg}
\end{align}
The first choice of geodesic, the {Euclidean Geodesic} \eqref{eq:eg}, corresponds to the distance functional $\|\bSigma_0 - \bSigma_1\|_F$. The second choice of geodesic, the {\BW\ Geodesic} \eqref{eq:bwg}, corresponds to the \BW\ distance functional $\dbw(\bSigma_0, \bSigma_1)$.


We note that \eqref{eq:bwg} is equivalent to
\begin{align}\label{eq:buresgeo}
    \bSigma_t &=(I + t(\bT-\bI))\bSigma_0(I + t(\bT - \bI)) \\
     &= (\bSigma_0^{1/2} + t \bDelta)(\bSigma_0^{1/2} + t \bDelta)^\top.
\end{align}
where $\bDelta = (\bT - \bI)\bSigma_0^{1/2}$. 

One of the big differences in these two paths is that while the Euclidean geodesics are linear in $t$, which is a result of the underlying flatness of the space, the \BW\ geodesic contains terms that are quadratic in $t$.
This points to the fact that this choice of geometry adds \emph{curvature} to the space $\bbS_+^d$. 


If we restrict ourselves to rank $r$ PSD matrices, $\bbS_+^{d, r}$, the geodesic \eqref{eq:bwg} becomes 
\begin{align}
   \bSigma_t &= (I + t(T - I)\bU_0 \bU_0^\top  (I + t(T - I))^\top \\ \nonumber
   &= ( (1-t) \bU_0 + t \bU_1)   ((1-t)\bU_0 + t \bU_1)^\top,
\end{align}
where we use the fact that $\bT \bSigma_0 \bT =\bT \bU_0 \bU_0^\top \bT = \bSigma_1 = \bU_1 \bU_1^\top$. Note that there is an inherent rotational symmetry in the problem, since for any $R \in \R^{r \times r}$ such that $\bR \bR^\top = \bI$ and $\bU \bU^\top \in \bbS_+^{d,r}$, $\bU \bU^\top = \bU \bR \bR^\top U^\top$.

\section{Supplementary Proofs}

\subsection{Proof of Propositions \ref{prop:equiv} and \ref{prop:whitenedbary}}

\begin{proof}[Proof of Proposition \ref{prop:equiv}]
	
First, if $\frac{1}{n} \sum_i \bx_i \bx_i^\top = \bI$, then
	$$\frac{1}{n} \sum_i \cA(\bSigma)_i = \Tr(\bSigma).$$
Indeed, this follows from the fact that
\begin{equation}
	\frac{1}{n} \sum_i \cA(\bSigma)_i = \frac{1}{n} \sum_i \bx_i^\top \bSigma \bx_i = \Tr\left( \bSigma \frac{1}{n} \sum_i \bx_i \bx_i^\top\right) = \Tr(\bSigma).
\end{equation}
With this in mind, we expand the square in \eqref{eq:lsmatrec_new}
\begin{align}
	&\frac{1}{2n} \| \sqrt{\cA(\bSigma)} - \sqrt{y}\|_2^2 =\\ \nonumber 
	&\frac{1}{2n} \sum_i \cA(\bSigma)_i + \frac{1}{2n} \sum_i y_{\bI} - \frac{1}{n}\sum_i  \sqrt{y_i} \sqrt{\bx_i^\top \bSigma \bx_i} \\ \nonumber
	&= \frac{1}{2} \Tr(\bSigma) -  \frac{1}{n}\sum_i  \sqrt{y_i} \sqrt{\bx_i^\top \bSigma \bx_i} + \frac{1}{2n} \sum_i y_i.
\end{align}
Thus, the minimization in \eqref{eq:lsmatrec_new} is equivalent to the program
\begin{equation}\label{eq:lsmatrec_new_2}
	\min_{\bSigma \in \bbS_{+}^d} \Tr(\bSigma) - \frac{2}{n} \sum_{i=1}^n \sqrt{y_i} \sqrt{\bx_i^\top \bSigma \bx_i}
\end{equation}

On the other hand, it is not hard to show that the \BW\ distance between a matrix $\bSigma \in \bbS_+^d$ and a rank one matrix $\wwt$ is
\begin{equation}
	\dbw^2(\bSigma, \wwt) = \Big[\Tr(\bSigma) + \Tr(\bx\bx^\top) - 2\sqrt{\Tr(\bx\bx^\top \bSigma)}\Big].
\end{equation}
Letting $\wwt =  y_i \bx_i \bx_i^\top$ and summing over $i$, we see that \eqref{eq:lsmatrec_new_2} is equivalent to minimizing \eqref{eq:bary} when 
$$ Q = \frac{1}{n} \sum_{i=1}^n \delta_{ y_i \bx_i \bx_i^\top}.$$
\end{proof}

\begin{proof}[Proof of Proposition \ref{prop:whitenedbary}]
	Consider the first order condition for the $n$ rank one matrices $y_i^2 \bC_n^{-1/2}  \bx_i \bx_i^\top \bC_n^{-1/2} $ given by
	\begin{equation}
		\frac{1}{n} \sum_{i=1}^n \frac{\Tr(\bx_i \bx_i^\top \bS)^{1/2}}{\Tr(\bC_n^{-1/2} \bx_i \bx_i^\top \bC_n^{-1/2} \bSigma_n)^{1/2}} \bC_n^{-1/2} \bx_i \bx_i^\top \bC_n^{-1/2} = \bI.
	\end{equation}
	Since $\bC_n^{-1/2} \bC_n \bC_n^{-1/2} = \bI$, we see that $\bC_n^{-1/2} \bSigma_n \bC_n^{-1/2} = \bS$ is a sufficient condition for $\bSigma_n$ to be a barycenter of 
	\begin{equation}
	    Q = \frac{1}{n} \sum_{i=1}^n \delta_{y_i \bC_n^{-1/2} \bx_i \bx_i^\top \bC_n^{-1/2}}
	\end{equation}
\end{proof}

\subsection{Proof of Theorem \ref{thm:main}}

The proof of Theorem \ref{thm:main} proceeds in the following sections. In Section \ref{subsec:ripcond}, we establish some restricted isometry properties (RIP) that will be used in our proof. Then, Section \ref{subsec:esc} proves that the function $F$ is Euclidean strongly convex over $\bbS_+^d$ with high probability. After this, Section \ref{subsec:esm} shows that $F$ also satisfies a certain Euclidean smoothness over rank $r$ matrices. Section \ref{subsec:foo} discusses first order optimality conditions for the barycenter, and in particular gives sufficient conditions for a point $\bSigmabar$ to be the minimizer of $F$. Section \ref{subsec:desclem} gives a descent lemma for the fixed-rank gradient descent method. After this, Section \ref{subsec:bwsc} proves local strong convexity of $F$ over fixed-rank PSD matrices in $\bbS_+^{d, r}$ that are close to $S$. We finish in Section \ref{subsec:mainproof} by putting all of these facts together.

\subsubsection{RIP Conditions}
\label{subsec:ripcond}

We discuss here a case an RIP condition that becomes essential for our later proof.
As discussed in \cite{cai2015rop}, issues arise in trying to prove a full $\ell_2$ RIP for this problem, due to the fact that the fourth moments of $\bx$ that show up. Instead, both \cite{cai2015rop,chen2015exact} prove the following $\ell_2/\ell_1$ RIP condition (also referred to as ``Restricted Uniform Boundedness").
\begin{thm}[\cite{chen2015exact} Proposition 1]
	\label{thm:RIP}
Suppose that $\bx_1, \dots, \bx_n$ are a sample from a sub-Gaussian distribution with $\E \bx_i = \bzero$, $\E x_{ij}^2 = 1$, and $\E x_{ij}^4 > 1$. Then, there are constants $c_1, c_2, c_3 > 0$ such that with probability exceeding $1-\exp(-c_3 n)$ 
\begin{equation}\label{eq:RIP}
    c_1 \| \bDelta\|_F \leq  \frac{1}{n} \left\| \begin{bmatrix} \bx_1^\top \bDelta \bx_1 - \bx_2^\top \bDelta \bx_2 \\ \bx_3^\top \bDelta \bx_3 - \bx_4^\top \bDelta \bx_4 \\ \vdots \\ \bx_{n-1}^\top \bDelta \bx_{n-1} - \bx_{n}^\top \bDelta \bx_{n} \end{bmatrix} \right \|_1 \leq c_2 \| \bDelta \|_F.
\end{equation}
hold simultaneously for all rank $r$ matrices $\bDelta$
provided that $n \gtrsim dr$.
\end{thm}
We have the following corollary to this theorem.
\begin{cor}\label{cor:rip}
	Suppose that a random vector $\bw = (w_1, \dots, w_d)^\top$ is sub-Gaussian with $\E \bw = \bzero$, $\E w_j^2 = 1$, $\E w_j^4 > 1$. Then, the following population RIP holds:
	\begin{equation}
		\E_{\bw} \Tr(\bw\bw^\top \bA)^2 \geq c_1 \|\bA\|_F^2.
	\end{equation}
	Furthermore, for a sample of $n$ i.i.d. copies of $W$, $\bw_1, \dots, \bw_n$, there are constants $c_2, c_3 > 0$ such that with probability exceeding $1-\exp(-c_3 n)$
	\begin{equation}
		\frac{1}{n} \sum_{i=1}^n \Tr(\bw_i \bw_i^\top \bA)^2 \geq c_2 \|\bA\|_F^2
	\end{equation}
	hold simultaneously for all rank $r$ matrices $\bDelta$
provided that $n \gtrsim dr$.
\end{cor}

\begin{proof}
	The first part holds from the argument in Appendix A of \cite{chen2015exact}.
	
	For the second part, we compute
	\begin{align}
		\frac{1}{n} \sum_{i=1}^n \Tr(\bw_i \bw_i^\top \bA)^2 &= \frac{1}{n} \sum_{i=1}^n \Tr(\bw_i \bw_i^\top \bA)^2 \\ \nonumber
		&= \frac{1}{n} \sum_{i=1}^{n/2} \Tr(\bw_{2i} \bw_{2i}^\top \bA)^2 + \Tr(\bw_{2i-1} \bw_{2i-1}^\top \bA)^2 \\ \nonumber
		&\geq \frac{1}{2n} \sum_{i=1}^{n/2} [\Tr(\bw_{2i} \bw_{2i}^\top \bA) - \Tr(\bw_{2i-1} \bw_{2i-1}^\top \bA)]^2 \\ \nonumber
		&= \frac{1}{n^2} \frac{n}{2} \sum_{i=1}^{n/2} [\Tr(\bw_{2i} \bw_{2i}^\top \bA) - \Tr(\bw_{2i-1} \bw_{2i-1}^\top \bA)]^2 \\ \nonumber
		&\geq \frac{1}{n^2} \Big[\sum_{i=1}^{n/2} |\Tr(\bw_{2i} \bw_{2i}^\top \bA) - \Tr(\bw_{2i-1} \bw_{2i-1}^\top \bA)|\Big]^2 \\ \nonumber
		&\geq c_2^2 \|\bA\|_F^2.
	\end{align}
\end{proof}

\subsubsection{Euclidean Strong Convexity}
\label{subsec:esc}

We begin by proving strong convexity along Euclidean geodesics. In particular, this guarantees that $S$ is the unique minimizer of $F$ over $\bbS_+^d$. Furthermore, it shows that $S$ is the unique stationary point.

First, it is easy to see that $S$ is stationary because 
\begin{equation}
	\nabla F(\bS) = \bI - \E_Q \sqrt{\frac{\bx^\top \bS \bx}{\bx^\top \bS \bx}}\bx\bx^\top  = \bI - \E_Q \bx\bx^\top = 0.
\end{equation}
Define the set 
\begin{equation}
	\cS(m, M) = \{\bSigma \in \bbS_+^d :  m \leq \lambda_r(\bSigma) \leq \lambda_1(\bSigma) \leq M \}.
\end{equation}
We have the following lemma over this set.
\begin{lem}\label{lem:popesc}
	Let $y = \langle \bx\bx^\top, \bS\rangle$, where $\bx \sim N(\bzero, \bI)$, $\bS$ is rank $r$, and $m, M$ be such that $\bS \in \cS(m, M)$. Then, the population obective $F$ is Euclidean strongly convex over $\cS(0, M)$ with constant $ \frac{c_1 \beta^2 d }{6  M^{3/2}}$.	
\end{lem}
In other words, the function $F(\bSigma_t)$ is strongly convex when $\bSigma_t$ is defined by \eqref{eq:eg}.
\begin{proof}

Let $\bSigma_t$ denote the geodesic between $\bSigma_0$ and $\bSigma_1$ given by $(1-t) \bSigma_0 + t \bSigma_1$.	
We compute the derivatives of the \BW\ distance as
\begin{align}
	\partial_t \dbw(\bSigma_t, \bx\bx^\top)^2 &= \Tr(\bSigma_1 - \bSigma_0)  \\ & - \frac{\Tr((\bSigma_1 - \bSigma_0)\bx\bx^\top)}{\sqrt{\bx^\top \bSigma_t \bx}},\\
	\partial_t^2 \dbw(\bSigma_t, \bx\bx^\top)^2 &= \frac{1}{2} \frac{\Tr(\xxt (\bSigma_1 - \bSigma_0))^2}{\Tr(\xxt \bSigma_t )^{3/2}} \label{eq:secondder}
\end{align}

By the definition of $M$,  we have the uniform bound
\begin{align}
	&\partial_t^2 F(\bSigma_t)|_{t=s} = \E_{\xxt \sim Q} \frac{1}{2} \frac{\Tr(\xxt (\bSigma_1 - \bSigma_0))^2}{(\Tr(\xxt \bSigma_s))^{3/2}} \\ \nonumber
	&= \E_{\xxt \sim Q} \frac{1}{2\|\bx\|^3} \frac{\Tr(\xxt (\bSigma_1 - \bSigma_0))^2}{(\Tr(\xxt \bSigma_s/\|\bx\|^2))^{3/2}} \\ \nonumber
	&\geq \frac{1}{2(2M)^{3/2}} \E_{\xxt \sim Q} \frac{1}{\|\bx\|^3} \Tr(\xxt (\bSigma_1 - \bSigma_0))^2.
\end{align}

Using the definition $\wwt = (\bx^\top \bS \bx) \xxt$,
\begin{align}
	\partial_t^2 F(\bSigma_t)|_{t=s} &\geq \frac{1}{6 M^{3/2}} \E_{\bx \sim N(\bzero, \bI)} \frac{\sqrt{\bx^\top \bS \bx}}{\|\bx\|^3} \Tr(\bx\bx^\top (\bSigma_1 - \bSigma_0))^2 \\ \nonumber
	&\geq \frac{1}{6 M^{3/2}} \E_{\bx \sim N(\bzero, \bI)} |X\|^2 \sqrt{\frac{\bx^\top \bS \bx}{ \|\bx\|^2}} \Tr(\frac{\bx\bx^\top}{\|\bx\|^2} (\bSigma_1 - \bSigma_0))^2 \\ \nonumber
	&= \frac{d}{6 M^{3/2}} \E_{\bx \sim N(\bzero, \bI)} \sqrt{ \frac{\bx^\top \bS \bx}{ \|\bx\|^2}} \Tr(\frac{\bx\bx^\top}{\|\bx\|^2} (\bSigma_1 - \bSigma_0))^2 .
\end{align}
Define the random vector
\begin{equation}\label{eq:wvec}
	W:= \Big(\frac{\bx^\top \bS \bx }{\|\bx\|^{2}}\Big)^{1/8} \frac{\bx}{\|\bx\|}  .
\end{equation}
By symmetry, $\bw$ is mean zero. Furthermore, $\bw$ is bounded and thus sub-Gaussian. Furthermore, it is a simple exercise to show that  
\begin{equation}
	\bC_W:= \E \wwt \succeq \beta \bI,
\end{equation}
for some dimension and $\bS$-dependent constant $\beta$. We can thus bound
\begin{align}
	  \E_{\bx} \Tr\left ( \wwt  (\bSigma_1 - \bSigma_0) \right)^2 &=\E_{W} \Tr\left ( \bC_W^{1/2} \bC_W^{-1/2} \wwt \bC_W^{-1/2} \bC_W^{1/2} (\bSigma_1 - \bSigma_0) \right)^2 \\ \nonumber
	  &\geq \beta^2 \E_W \Tr(\bC_W^{-1/2} \wwt \bC_W^{-1/2} (\bSigma_1 - \bSigma_0))^2.
\end{align}
Using the fact that $\bC_W^{-1/2} \bw$ is mean zero sub-Gaussian with identity covariance (and furthermore the 4th moment condition is trivially satisfied), we can now use the RIP condition of Corollary \ref{cor:rip} to bound
\begin{equation}
	\E_{\wwt  \sim Q} \Tr(\bC_W^{-1/2} \wwt \bC_W^{-1/2} (\bSigma_1 - \bSigma_0 ))^2 \geq c_1 \|\bSigma_1 - \bSigma_0\|_F^2.
\end{equation}

Putting this all together, 
\begin{align}
	\partial_t^2 F(\bSigma_t)|_{t=s} &\geq \frac{c_1 \beta^2 d}{2(2M)^{3/2}} \|\bSigma_1 - \bSigma_0\|_F^2 \\ \nonumber
	&\geq \frac{c_1 \beta^2 d}{6 M^{3/2}} \|\bSigma_1 - \bSigma_0\|_F^2.
\end{align}

\end{proof}

The sample setting requires a bit more work. In this case, we observe ${y}_i$ and $\bx_i$, $i=1, \dots, n$, and to recover the matrix $S$ we first compute the barycenter of ${y}_i \bC_n^{-1/2} \bx_i \bx_i^\top \bC_n^{-1/2}$, where $\bC_n$ is the sample covariance. We denote this whitened discrete distribution by $\tilde Q$.  In this case, by Proposition \ref{prop:whitenedbary}, the barycenter is $\bSigmabar = \bC_n^{-1/2} S \bC_n^{-1/2}$
\begin{lem}\label{lem:sampesc}
With probability at least $1-\exp(-c_2 n)$ for some constant $c_2$, $F$ is Euclidean strongly convex over $\cS(0,M)$ with constant $\frac{c_3 \beta^2 d}{6 M^{3/2}}$. 
\end{lem}

\begin{proof}
In this case, 
\begin{align}
	\partial_t^2 F(\bSigma_t)|_{t=s} &\geq \frac{1}{6 M^{3/2}} \frac{1}{n} \sum_{i=1}^n \frac{\sqrt{\bx_i^\top \bS \bx_i}}{\|\bC_n^{-1/2} \bx_i\|^3} \Tr(\bC_n^{-1/2} \bx_i \bx_i^\top \bC_n^{-1/2} (\bSigmabar - \bSigma))^2 \\ \nonumber
	&\geq  \frac{1}{6 M^{3/2}}  \frac{\lambda_{\min}(\bC_n)^{3/2}}{\lambda_{\max}(\bC_n)^2} \frac{1}{n} \sum_{i=1}^n \frac{\sqrt{\bx_i^\top \bS \bx_i}}{\| \bx_i\|^3} \Tr( \bx_i \bx_i^\top (\bSigmabar - \bSigma))^2 \\ \nonumber
	&=  \frac{1}{6 M^{3/2}}  \frac{\lambda_{\min}(\bC_n)^{3/2}}{\lambda_{\max}(\bC_n)^2} \frac{1}{n} \sum_{i=1}^n \Tr\Big( \|\bx_i\|\Big(\frac{\bx_i^\top \bS \bx_i}{\|\bx_i\|^2} \Big)^{1/4} \frac{\bx_i \bx_i^\top}{\|\bx_i\|^2} (\bSigmabar - \bSigma)\Big)^2 .
\end{align}
Then, following the same line of reasoning as in the previous proof to obtain strong convexity the sub-Gaussianity of the random vector
\begin{equation}
	\bw_i = \|\bx_i\|^{1/2} \Big(\frac{\bx_i^\top \bS \bx_i}{\|\bx_i\|^2} \Big)^{1/8} \frac{\bx_i }{\|\bx_i\|},
\end{equation}
which again a simple exercise shows $\E \bw_i \bw_i^\top \succeq \beta \bI$ and $\E \bw_i = 0$.
Then, with probability at least $1-\exp(-c_2 n)$, the discrete distribution is strongly convex with constant $\frac{c_3 \beta^2 d}{6 M^{3/2}}$. 
\end{proof}

Notice that in both the population and the sample setting with high probability we get a unique barycenter $S$.


\subsubsection{Local Euclidean Smoothness}
\label{subsec:esm}

We remind ourselves that
\begin{equation}
	\nabla F(\bSigma) = \bI - \E_Q \sqrt{\frac{\bx^\top \bS \bx}{\bx^\top \bSigma \bx}} \bx\bx^\top,
\end{equation}

Using this, we have the following lemma.
\begin{lem}
	Let $y = \langle \bx\bx^\top, \bS$, $X \sim N(\bzero, \bI)$, $\rank(\bSigma) \geq 3$, and $\lambda_r(\bSigma) \geq m$. Then, 
	\begin{equation}
		\left\| \nabla F(\bS) -  \nabla F(\bSigma) \right\|_F = \left\| \nabla F(\bSigma) \right\|_F \leq \frac{d^{3/2}}{m\sqrt{r}}\|\bSigma - \bS\|_F
	\end{equation}
\end{lem}
\begin{proof}
We have that
\begin{align}
	\left\| \E_{Q} \sqrt{\frac{\bx^\top \bS \bx}{\bx^\top \bS \bx}} \bx\bx^\top -  \E_{Q} \sqrt{\frac{\bx^\top \bS \bx}{\bx^\top \bSigma \bx}} \bx\bx^\top \right\| &\leq \E_{Q}  \left|\frac{\sqrt{\bx^\top \bSigma \bx} - \sqrt{\bx^\top \bS \bx}}{\sqrt{\bx^\top \bSigma \bx}}\right|   \|\bx\|^2 \\ \nonumber
	&= d \E_{Q} \frac{1}{\sqrt{\bx^\top \bSigma \bx}} \sqrt{|\bx^\top (\bSigma - \bS) \bx|}\\ \nonumber
	&\leq d \|\bSigma - \bS\|_F^{1/2} \E_{Q} \frac{\|\bx\|}{\sqrt{\bx^\top \bSigma \bx}} 
\end{align}
As long as $r \geq 3$ and $\lambda_r(\bSigma) \geq m$, we have that
\begin{equation}
	 \E_{Q} \frac{\|\bx\|}{\sqrt{\bx^\top \bSigma \bx}}  \leq \frac{1}{m}  \sqrt{1+\frac{d-r}{r}}= \frac{\sqrt{d}}{m\sqrt{r}}. 
\end{equation} 

\end{proof}
Thus, as $\bSigma \to S$ for $\rank(\bSigma) \geq 3$, $\|\nabla F(\bSigma)\|_F \to 0$.

\subsubsection{First Order Optimality of the Low-Rank Barycenter}
\label{subsec:foo}

Moving on to the \BW\ geometry, we first show that the low-rank barycenter is a first-order stationary point with respect to \BW\ distance. While this follows from the previous results due to the fact that it is a global minimum over $\bbS_+^d$, we take a different approach here based on the fixed point iteration of~\cite{agueh2011barycenter}. This fixed point iteration forms the basis of our efficient low-rank algorithm.

By \cite{agueh2011barycenter}, a sufficient condition for $ \gamma_{\bSigma}$ to solve \eqref{eq:bary}, is
\begin{equation}\label{eq:1dident}
    \E_{\bx\bx^\top \sim Q} \frac{\xxt}{\Tr(\bx \bx^\top \bSigma)^{1/2}} = \bI,
\end{equation}
where $\bI$ is the identity matrix in $\R^d$. Notice that this corresponds to the gradient $\nabla F$ being equal to $\bzero$. In the following, we let
\begin{equation}
    \tilde{T}(\bSigma):= \E_Q \frac{\xxt}{\Tr(\xxt \bSigma)^{1/2}} 
\end{equation}

Consider the map corresponding to gradient descent with step size 1,
\begin{equation}
	\bSigma_{k+1} = \Big(\tilde{T}(\bSigma_k) \Big) \bSigma_k \Big(\tilde{T}(\bSigma_k)\Big).
\end{equation}
The corresponding fixed point equation is
\begin{equation}\label{eq:implementedfixedpoint}
	\bSigma = \Big( \tilde{T}(\bSigma) \Big) \bSigma \Big(\tilde{T}(\bSigma)\Big).
\end{equation}
\citet{chewi2020gradient} prove that the operator norm $\|\cdot\|_2$ is convex along generalized geodesics. Therefore,\eqref{eq:implementedfixedpoint} maps a compact subset of $\mathbb{S}_+$ to itself, and one can apply the Brouwer fixed point theorem to guarantee a solution. 
The fixed point satisfies a restricted first order condition given by
\begin{equation}\label{eq:restrfoc}
    \bP_{\Sp(\bSigma)} \tilde{T}(\bSigma) \bP_{\Sp(\bSigma)}= \bP_{\Sp(\bSigma)} = \id_{\Sp(\bSigma)} = \bP_{\Sp(\bSigma)}.
\end{equation}
Notice that if $\bSigma$ is full rank, then this implies \eqref{eq:1dident}. More generally, we need an extra condition on top of first order optimaltiy to guarantee that $\bSigma$ is a barycenter. 

\begin{pro}[Sufficient Condition for Barycenter]\label{}
	If $\bSigma$ satisfies the first order conditions 
	\begin{align}\label{eq:foc11}
		\bP_{\bSigma} \tilde{T}(\bSigma) \bP_{\bSigma}&= \bP_{\bSigma}, \\
		\tilde{T}(\bSigma) &\preceq \bI.\label{eq:foc21}
	\end{align}
	then $\bSigma$ is a barycenter of $Q$. Furthermore, in our observation model where $Q$ is the law of $\sqrt{\bx^\top \bS \bx} \bx\bx^\top$, for $\bx \sim N(\bzero, \bI)$, $\bS$ satisfies \eqref{eq:foc11} and \eqref{eq:foc21}.
\end{pro}

\begin{proof}
    
The first equation, \eqref{eq:foc11}, guarantees that $\bSigma$ is a fixed point satisfying \eqref{eq:implementedfixedpoint}. On the other hand, \eqref{eq:foc21} guarantees that all directional derivatives are positive (and that $\tilde{T}(\bSigma) \preceq \bI$), and thus $\bSigma$ is a local minimum. To see this, suppose that $\bSigma = \bSigma_0$ is a fixed point and $\bSigma_1$ is another PSD matrix.
The directional derivatives are
\begin{align}
    \partial_t F(\bSigma_t) |_{t=0} &= \Tr\Big( (\bI - \tilde T(\bSigma)) (\bSigma_1 - \bSigma_0 )\Big) \\ \nonumber
    &= \Tr\Big( (\bI - \tilde T(\bSigma)) \bSigma_1  \Big).
\end{align}
If $\tilde T(\bSigma) \not \preceq \bI$, then there is a $\bSigma_1$ such that this is less than zero, and therefore $\bSigma_0$ cannot be a minimum. On the other hand, if $\tilde T(\bSigma) \preceq \bI$, then all directional derivatives are positive. Combined with the Euclidean convexity result in the paper, this proves that $\bSigma$ would be the global minimum and thus the barycenter.

Finally, in our observation model, we have 
\begin{equation}
    \tilde{T}(\bS) = \bI,
\end{equation}
and so both \eqref{eq:foc11} and \eqref{eq:foc21} hold.
\end{proof}

Notice alternatively that this also implies that the barycenter is the only stationary point in the set where $\tilde{T}(\bSigma) \succeq \bI$. In particular, this is because at all points where $\tilde{T}(\bSigma) \neq \bI$, one can find a direction of decrease.

\subsubsection{Smoothness and a Descent Lemma}
\label{subsec:desclem}

The barycenter functional is \emph{smooth}, as is shown in~\cite{chewi2020gradient}. This result extends to non absolutely continuous measure by noting that 1) nonnegative curvature extends to measures that are not absolutely continuous (with respect to Lebesgue, see~\cite{ambrosio2008gradient} Lemma 7.3.2), and 2) the characterization of the derivative of the Wasserstein distance extends to cases where the measures are not absolutely continuous (see \cite{ambrosio2008gradient} Lemma 7.3.6).

With smoothness, we have the following descent lemma over fixed rank \BW\ space. Note that such a descent lemma is standard in the analysis of gradient descent methods (see \citet[Theorem 2.1.5]{nesterov2004introductory}). Here, $\bSigma^+$ is the update after one gradient step.
\begin{lem}\label{lem:descent}
	Any $\bSigma \in \bbS_+^d$ and $\bSigma^+ = \tilde{T}(\bSigma) \bSigma \tilde{T}(\bSigma)$, it holds that
	\begin{equation}
		F(\bSigma^+) - F(\bSigma) \leq -\frac{1}{2} \left\| \bI - \tilde{T}(\bSigma) \right\|_{\gamma_{\bSigma}}^2.
	\end{equation}
\end{lem}

\subsubsection{Local Geodesic Strong Convexity}
\label{subsec:bwsc}

We now prove local strong convexity in the population setting. By the previous section, we know that $S$ is the unique barycenter of $Q$ since the variance inequality holds for arbitrarily small $m$ and $\rho$ and arbitrarily large $M$.
\begin{pro}\label{prop:bwsc}
Let $y = \langle \bx\bx^\top , \bS\rangle$, $X \sim N(\bzero, \bI)$, $\rank(\bS) \geq 3$, and $m, M$ be such that $m < \lambda_r(\bS) \leq \lambda_1(\bS) < M$. Then, if
\begin{equation}
	\rho \leq \frac{m^2 r}{d^3} \frac{c_1^2 \beta^4 m^2}{36  M^{3}}
\end{equation}
$F$ is locally geodesically strongly convex over 
\begin{equation}
	\cS(m, M) \cap \{\bSigma \in \bbS_+^{d,r} : \|\bSigma - \bS\|_F \leq \rho\}.
\end{equation}
\end{pro}

\begin{proof}

Fix $\bSigma_0, \bSigma_1 \in  \cS(m,M) \cap \{\bSigma \in \bbS_+^{d,r} : \dbw^2(\bSigma, \bSigmabar) \leq \rho\} $.
	
	Suppose there exists a transport map $T$ between $\bSigma_0$ and $\bSigma_1$, which we are guaranteed for $\rho$ sufficiently small. Then for the \BW\ geodesic $\bSigma_t = ((1-t)\bI +t\bT)\bSigma_0((1-t)\bI + t\bT)$, we can compute
\begin{align}
	\partial_t^2 F(\bSigma_t)|_{t=s} &= \E_{\xxt \sim Q} \Tr[(\bT - \bI)\bSigma_0(\bT - \bI)]+ 2\frac{\Tr(\xxt (\bT - \bI)\bSigma_0)^2}{(\Tr(\xxt \bSigma_s))^{3/2}} \\ \nonumber
	&- \Tr \left[ \frac{\xxt (\bT - \bI) \bSigma_0(\bT - \bI)}{\sqrt{\Tr(\xxt \bSigma_s)}} \right].
\end{align}
Some manipulation when $s=0$ yields
\begin{align}\label{eq:secder_bwlb}
	\partial_t^2 F(\bSigma_t)|_{t=0} &= \Big \langle (\bT - \bI) \bSigma_0 (\bT - \bI), \nabla F(\bSigma_0) \Big \rangle + \E_{\xxt \sim Q}2\frac{\Tr(\xxt (\bT - \bI)\bSigma_0)^2}{(\Tr(\xxt \bSigma_0))^{3/2}} \\ \nonumber
	&\geq -\dbw^2(\bSigma_0, \bSigma_1) \|\nabla F(\bSigma_0)\|_F + \E_{\xxt \sim Q}2\frac{\Tr(\xxt (\bT - \bI)\bSigma_0)^2}{(\Tr(\xxt \bSigma_0))^{3/2}}. 
\end{align}
The last term satisfies the lower bound
\begin{align}
	\E_{\xxt \sim Q}2\frac{\Tr(\xxt (\bT - \bI)\bSigma_0)^2}{(\Tr(\xxt \bSigma_0))^{3/2}} &\geq \frac{c_1 \beta^2 }{3  M^{3/2}} \|(\bT - \bI)\bSigma_0\|_F^2 \\ \nonumber
	&\geq \frac{c_1 \beta^2 m}{3  M^{3/2}}\|(\bT - \bI)\bSigma_0^{1/2}\|_F^2 \\ \nonumber
	&= \frac{c_1 \beta^2 m}{3  M^{3/2}} \dbw^2(\bSigma_0, \bSigma_1)
\end{align}

On the other hand, by our Euclidean arguments, $\bS$ is the unique point such that $\nabla F = \bzero$. We can upper bound $\|\nabla F(\bSigma)\|_F$ on the set $\{\bSigma : \|\bSigma - \bS\| \leq \rho\}$ by
\begin{equation}
    \|\nabla F(\bSigma)\|_F \leq \frac{d^{3/2}}{m\sqrt{r}} \sqrt{\rho}.
\end{equation}

Thus, if 
\begin{equation*}
    \rho \leq \frac{m^2 r}{d^3} \frac{c_1^2 \beta^4 m^2}{36  M^{3}}
\end{equation*}
then 
\begin{equation}\label{eq:locsc}
    \partial_t^2 F(\bSigma_t)|_{t=0} \geq \frac{c_1 \beta^2 m}{6  M^{3/2}} \dbw^2(\bSigma_0, \bSigma_1).
\end{equation}
for all $\bSigma_0 \in \{\bSigma : \|\bSigma - \bS\|_F \leq \rho\}$. Noticing the fact that $\dbw(\bSigma, \bS) \leq \|\bSigma - \bS\|_F$, we have local strong geodesic convexity in a ball around $\bS$.

\end{proof}

We note that the same proof extends this to the sample setting with high probability in an analogous way to how Lemma \ref{lem:sampesc} extends Lemma \ref{lem:popesc} to the sample setting.

\subsubsection{Proof of Theorem \ref{thm:main}}
\label{subsec:mainproof}

Suppose that we initialize such that $m \leq \lambda_r(\bU_0 \bU_0^\top) \leq \lambda_1(\bU_0\bU_0^\top) \leq M$, and 
\begin{align}
    F(\bU_0\bU_0^\top) - F(\bS) &\leq \Big( \frac{m^2 r}{d^3} \frac{c_1^2 \beta^4 m^2}{36  M^{3}}\Big)^2 \cdot \frac{c_1 \beta^2 d}{6M^{3/2}} \\ \nonumber
    &= \frac{c_1^{5}m^8r^2 \beta^{10}}{6^5 M^{15/2}d^5}.
\end{align} 
First, by Euclidean strong convexity, we have
\begin{equation}
	\|\bSigma - \bS\|_F \leq \sqrt{(F(\bSigma) - F(\bS)) \frac{6M^{3/2}}{c_1 \beta^2 d}}.
\end{equation}
Therefore, the initialization condition on $F(\bU_0 \bU_0^\top) - F(\bS)$ is enough to guarantee that 
\begin{align}
	\|\bU_0 \bU_0^\top - \bS\| \leq  \frac{m^2 r}{d^3} \frac{c_1^2 \beta^4 m^2}{36  M^{3}}.
\end{align}
Furthermore, by Lemma \ref{lem:descent}, $F(\bU_k \bU_k^\top) \leq F(U_{k-1}U_{k-1}^\top)$ for all $k$, and thus
\begin{equation}
	\|\bU_k \bU_k^\top - \bS\| \leq  \frac{m^2 r}{d^3} \frac{c_1^2 \beta^4 m^2}{36  M^{3}}
\end{equation} 
for all $k$, and the iterates remain in the ball of strong geodesic convexity.

Using the strong geodesic convexity of Proposition \ref{prop:bwsc} and Lemma \ref{lem:descent}, it is then a standard argument to show linear convergence. We can apply, for example, \cite[Theorem 15]{zhangSra16a} to yield the result.

\subsection{Proof of Theorem \ref{thm:BWSGD}}
\label{sec:sgdconv}

We give a proof of this theorem by following \cite{ghadimi2013stochastic}.

Let $T_{{\bSigma_0} \to \bSigma_1} = T_{\gamma_{\bSigma_0} \to \gamma_{\bSigma_1}}$.
Let 
\begin{align}
	\bSigma_1 = &\Big((1-\eta) I + \eta \frac{\bx_1 \bx_1^\top}{\Tr(\bx_1 \bx_1^\top \bSigma_0)^{1/2}}\Big) \bSigma_0\\ \nonumber
	&\cdot \Big((1-\eta) I + \eta \frac{\bx_1 \bx_1^\top}{\Tr(\bx_1 \bx_1^\top \bSigma_0)^{1/2}}\Big).
\end{align}
or
\begin{equation}
	\gamma_{\bSigma_1} = [(1-\eta) \id + \eta T_{{\bSigma_0} \to {\bx_1 \bx_1^\top}}]_{\#} \gamma_{\bSigma_0}.
\end{equation}

Due to nonnegative curvature, we have
\begin{align}
	W_2^2(\gamma_{\bSigma_1}, \gamma_{\bx\bx^\top}) &\leq \|T_{\bSigma_0 \to \bSigma_1} - T_{\bSigma_0 \to \bx\bx^\top}\|_{\bSigma_0}^2 \\ \nonumber
	&= \| (1-\eta) \bI + \eta T_{{\bSigma_0} \to {\bx_1 \bx_1^\top}} - T_{{\bSigma_0} \to {\xxt}}\|_{\bSigma_0}^2 \\ \nonumber
	&= \| (\bI -T_{{\bSigma_0} \to {\xxt}}) + \eta (T_{{\bSigma_0} \to {\bx_1 \bx_1^\top}} -  \bI ) \|_{\bSigma_0}^2 \\ \nonumber
	&= \| \bI -T_{{\bSigma_0} \to {\xxt}}\|_{\bSigma_0}^2 + \eta^2 \| \bI -T_{{\bSigma_0} \to {\bx_1 \bx_1^\top}}\|_{\bSigma_0}^2 - 2 \eta \langle \bI -T_{{\bSigma_0} \to {\xxt}}, \bI -T_{{\bSigma_0} \to {\bx_1 \bx_1^\top}} \rangle .
\end{align}
Taking the expectation with respect to $\bx\bx^\top$,
\begin{align}
	&F(\bSigma_1) - F(\bSigma_0) \leq - 2 \eta \langle \nabla F(\bSigma_0), \bI -T_{{\bSigma_0} \to {\bx_1 \bx_1^\top}} \rangle_{\bSigma_{0}} \\ \nonumber
	&+ \eta^2 \dbw^2(\bSigma_0, \bx_1 \bx_1^\top).
\end{align}
Summing over iterations, we find
\begin{align}
	&F(\bSigma_K) - F(\bSigma_0) \leq \\ \nonumber
	&- 2 \eta \sum_{k=1}^K \langle \nabla F(\bSigma_{k-1}), \bI -T_{{\bSigma_0} \to {\bx_k \bx_k^\top}} \rangle_{\bSigma_{k-1}} \\ \nonumber
	&+ \eta^2 \sum_{k=1}^K \dbw^2(\bSigma_{k-1}, \bx_k \bx_k^\top).
\end{align}
Taking the expectation (properly) with respect to $\bx_1 \bx_1^\top, \dots, \bx_k \bx_k^\top$, and using $\E [\dbw^2(\bSigma_{k-1}, \bx_k \bx_k^\top) |\bSigma_{k-1} ]\leq \bSigma^2$,
\begin{equation}
	\E F(\bSigma_K) - F(\bSigma_0) \leq - 2 \eta \sum_{k=1}^K \E  \| \nabla F(\bSigma_k)\|^2_{\bSigma_k} + \eta^2 K \bSigma^2.
\end{equation}
With this,
\begin{equation}
	 \sum_{k=1}^K \E  \| \nabla F(\bSigma_k)\|^2_{\bSigma_k} \leq \E \frac{F(\bSigma_K) - F(\bSigma_0)}{2 \eta} + \frac{\eta}{2} K \bSigma^2.
\end{equation}
Choosing $\eta = 1/\sqrt{K}$,
\begin{equation}
	 \frac{1}{K}\sum_{k=1}^K \E  \| \nabla F(\bSigma_k)\|^2_{\bSigma_k} \leq \E \frac{F(\bSigma_K) - F(\bSigma_0)}{2 \sqrt{K}} + \frac{\eta}{2 \sqrt{K}} .
\end{equation}
Within the sequence of iterates $\bSigma_0, \dots, \bSigma_K$, at least one element must have gradient bounded by $O(1/\sqrt{K})$. 


\subsection{Suboptimal Stationary Points}

There potentially exist stationary points that are not optimal in the low-rank case. In particular, with the parametrization $\bSigma = \bU \bU^\top$, these are points such that
\begin{equation}
	\E \frac{\sqrt{\bx^\top \bV\bV^\top \bx}}{\sqrt{\bx^\top \bU \bU^\top \bx}} \bx\bx^\top \bU = \bU,
\end{equation}
where $\bV \bV^\top = \bS$.

In the following, we will show that at least in the $r=1$ case, there are no local minima $\bu \bu^\top$ that are not orthogonal to $\bv\bv^\top$.


\subsubsection{No Local Minima in 1D Case}

\begin{thm}
	Consider the observation model with the rank one matrix $\bS = \bv\bv^\top$  and $y = \langle \bx\bx^\top, \bS \rangle$. Then, the only fixed points of the iteration
	\begin{equation}
		\E \sqrt{\frac{\bx^\top \bv\bv^\top \bx}{\bx^\top \bu \bu^\top \bx}} \bx\bx^\top \bu = \bu
	\end{equation}
	are orthogonal to $\bv$ or $\bu=\bv$. Since the points orthogonal to $\bv$ are local maxima, population gradient descent converges to $\bv$. 
\end{thm}

\begin{proof}
	
In the 1-dimensional case, the stationary points are 
\begin{equation}
	\E \frac{|\bv^\top \bx|}{|\bu^\top \bx|} \bx\bx^\top \bu = \bu.
\end{equation}
Obviously, $\bu=\bv$ is a stationary point. 

On the other hand, suppose that $u \perp v$. Then, we can write
\begin{align*}
		\E \frac{|\bv^\top \bx|}{|\bu^\top \bx|} \bx\bx^\top \bu &= \E \frac{|\bv^\top \bx|}{|\bu^\top \bx|} (\bu \bu^\top /\|\bu\|^2 + \bv\bv^\top/\|\bv\|^2 + \wwt)\bx\bx^\top \bu \\ \nonumber
		&= \frac{1}{\|\bu\|^2}\E \frac{|\bv^\top \bx|}{|\bu^\top \bx|} \bu \bu^\top \xxt \bu + \frac{1}{\|\bv\|^2}E \frac{|\bv^\top \bx|}{|\bu^\top \bx|} \bv\bv^\top \bx \bx^\top \bu \\ \nonumber
		&= \Big(\frac{1}{\|\bu\|^2}\E |\bv^\top \bx| \E  |\bu^\top \bx|\Big) \bu + \frac{1}{\|\bv\|^2}\Big(E (\bv^\top \bx)^2  \E  \frac{\bx^\top \bu}{|\bu^\top \bx|} \Big)\bv \\ \nonumber
		&= \frac{2}{\pi} \|\bv\|  \frac{\bu}{\|\bu\|}.
\end{align*}
Therefore, for this to be a stationary point, we need
\begin{equation}
	\frac{2}{\pi} \frac{\|\bv\|}{\|\bu\|} = 1, \text{ or } \|\bu\| = \frac{2 \|\bv\|}{\pi}.
\end{equation}
Thus, any orthogonal vector with this length is a stationary point.

Finally, we show that there are no other stationary points.
For $u$ to be a fixed point, we need to have that
\begin{equation}\label{eq:1dfp}
	\E |\bv^\top \bx| \sign(\bu^\top \bx) \bx = u.
\end{equation}
Due to the rotational symmetry of the $\bx$'s, we can assume without loss of generality that only the first two coordinates of $\bv$, $\bu$ are nonzero. Furthermore, we can reduce to the two dimensional case, since the coordinates $x_3, \dots, x_d$ do not contribute to the expectation. 
Finally, we can assume without loss of generality that $u$ and $v$ are rotated so that $\bu =[a,0]^\top$. In this way, if $\bv = (v_1, v_2)$, \eqref{eq:1dfp} becomes
\begin{equation}
	\E |v_1 x_1 + v_2 x_2| \sign(x_1) [x_1, x_2]^\top.
\end{equation}
The first coordinate is obviously positive. If we can show that the second coordinate is nonzero, then $\bu$ cannot be a fixed point. We assume without loss of generality that $v_2 > 0$.

Thus we consider
\begin{equation}
	\E |v_1 x_1 + v_2 x_2|\sign(x_1) x_2
\end{equation} 
for i.i.d.~$N(0,1)$ random variables $x_1, x_2$. By symmetry, $(x_1, x_2)$ occurs with the same probability as $(-x_1, -x_2)$, and 
\begin{align}
	&|v_1 (-x_1) + v_2 (-x_2)|\sign(-x_1) (-x_2) =|v_1 x_1 + v_2 x_2|\sign(x_1) x_2.
\end{align}
Therefore, we can integrate over any half-plane, and so
\begin{align}
	&\E |v_1 x_1 + v_2 x_2|\sign(x_1) x_2 =\E_{(x_1, x_2)|x_1 > 0} |v_1 x_1 + v_2 x_2| x_2.
\end{align} 
For all fixed $x_1>0$, it is easy to see that 
\begin{equation}
	\E_{x_2|x_1 > 0} |v_1 x_1 + v_2 x_2| x_2 > 0,
\end{equation}
since $|v_1 x_1 + |v_2 x_2|| > |v_1x_1 - |v_2 x_2||$ when $v_2 > 0$.
Therefore the second coordinate cannot be zero, and this means that $u$ is not a fixed point.

Together with the monotonicity of gradient descent, which implies convergence to a fixed point, we conclude that population gradient descent in the 1D case converges to the underlying vector $\bv$.
\end{proof}


\subsection{Highly Local Recovery in Discrete 1D Case}

Suppose that we have a  discrete set of sensing vectors $\bx_1, \dots, \bx_n$ that satisfy the $\ell_2/\ell_1$-RIP condition, and that $\frac{1}{n} \sum_i \bx_i \bx_i^T = \bI$. Suppose that $\bS = \bv \bv^\top$. Define the set
\begin{equation}
    \cB = \{\bu : |\bu^\top \bx_i| > \epsilon \ \forall i\}.
\end{equation}
This is an open set. Over this set, we can bound
\begin{align}
    \|\nabla F(\bu) - \nabla F(\bv)\| &\lesssim \|\bv \bv^\top - \bu \bu^\top \|_F^{1/2} \frac{1}{n} \sum_i \frac{\|\bx_i\|}{|\bu^\top \bx_i|} \\ \nonumber
    &\leq C \|\bv \bv^\top - \bu \bu^\top \|_F^{1/2},
\end{align}
for some $C$ that depends on all parameters. Assume that $\bv \in \cB$. This implies that there exists a ball around $\bv$ that is contained within $\cB$. In this ball, we can get the local Euclidean smoothness bound given in Section \ref{subsec:esm}. In turn, this implies local geodesic strong convexity over a small subset of this ball, which implies local linear convergence.

\subsection{Lack of Strong Geodesic Convexity}

We illustrate here that the functional $F$, while being locally strongly convex about $\bS$ when restricted to rank $r$ matrices, is not locally strongly convex around $\bS$ for higher rank matrices.

Let $\bS$ be the matrix
\begin{equation}
    \bS = \begin{bmatrix} a & 0 \\ 0 & 0 \end{bmatrix},
\end{equation}
$\bSigma_0$ be the matrix
\begin{equation}
    \bSigma_0 = \begin{bmatrix} a & 0 \\ 0 & b \end{bmatrix},
\end{equation}
and let 
\begin{equation}
	\bT = \begin{bmatrix} 1 & 0 \\ 0 & 0\end{bmatrix}
\end{equation}
be the transport map from $\bSigma_0$ to $\bS$.
Let $\bSigma_t$ be the geodesic from $\bS$ Using the first display in \eqref{eq:secder_bwlb}, we find
\begin{align}
    \partial_t^2 F(\bSigma_t)|_{t=0} &= 
    \Big \langle (\bT - \bI) \bSigma_0 (\bT - \bI), \nabla F(\bSigma_0) \Big \rangle + \E_{\xxt \sim Q}2\frac{\Tr(\xxt (\bT - \bI)\bSigma_0)^2}{(\Tr(\xxt \bSigma_0))^{3/2}} \\ \nonumber
    &= \Big \langle \begin{bmatrix} 0 & 0 \\ 0 & b \end{bmatrix}, \nabla F(\bSigma_0) \Big \rangle + \E_{\xxt \sim Q}2\frac{\Tr(\xxt (\bT - \bI)\bSigma_0)^2}{(\Tr(\xxt \bSigma_0))^{3/2}}
\end{align}
By a trace inequality, 
 \begin{align}
	  2 \E_Q \frac{(\Tr[\bx\bx^\top (\bT - \bI)\bSigma_0^{1/2}])^2)}{\Tr[\bx\bx^\top \bSigma_0 ]^{3/2}} &\lesssim \|(\bT - \bI)\bSigma_0\|_F^2.
 \end{align}
To have geodesic strong convexity, we need to lower bound $\partial_t^2 F(\bSigma_t)|_{t=0}$ by $c \|(\bT - \bI) \bSigma_0^{1/2}\|_F^2 = c \dbw(\bSigma_0, \bSigma_1)^2$ for some $c>0$.
On the other hand, using the result of Section \ref{subsec:esm}, $\|\nabla F(\bSigma_0)\|_F \lesssim \|\bSigma_0 - \bS\|_F = |b|$, and so
\begin{equation}
    \partial_t^2 F(\bSigma_t)|_{t=0} \lesssim  b^2 + \|(\bT - \bI)\bSigma_0\|_F^2.
\end{equation}

We note that
\begin{align*}
	\|(\bT - \bI)\bSigma_0\|_F^2 &=  b^2.
\end{align*}
On the other hand, 
\begin{align*}
	\dbw^2(\bSigma_0, \bSigma_1) &= \|(\bT - \bI) \bSigma_0^{1/2}\|_F^2 = b, 
\end{align*}
There is no $c > 0$ such that
\begin{equation}
    b^2 \geq c b
\end{equation}
for all $b > 0$. Therefore, $F$ is not strongly geodesically convex at full-rank $\bSigma_0$ that are close the boundary. In particular, if the true barycenter is low-rank, then as the full-rank $\bSigma$ approach $\overline{\bSigma}$, we cannot expect strong convexity.

\section{Supplemental Experiments}


\subsection{Convergence of \BW\ gradient descent Versus \BW SGD}

\begin{figure}[h]
    \centering
    \includegraphics[width = 0.40\textwidth]{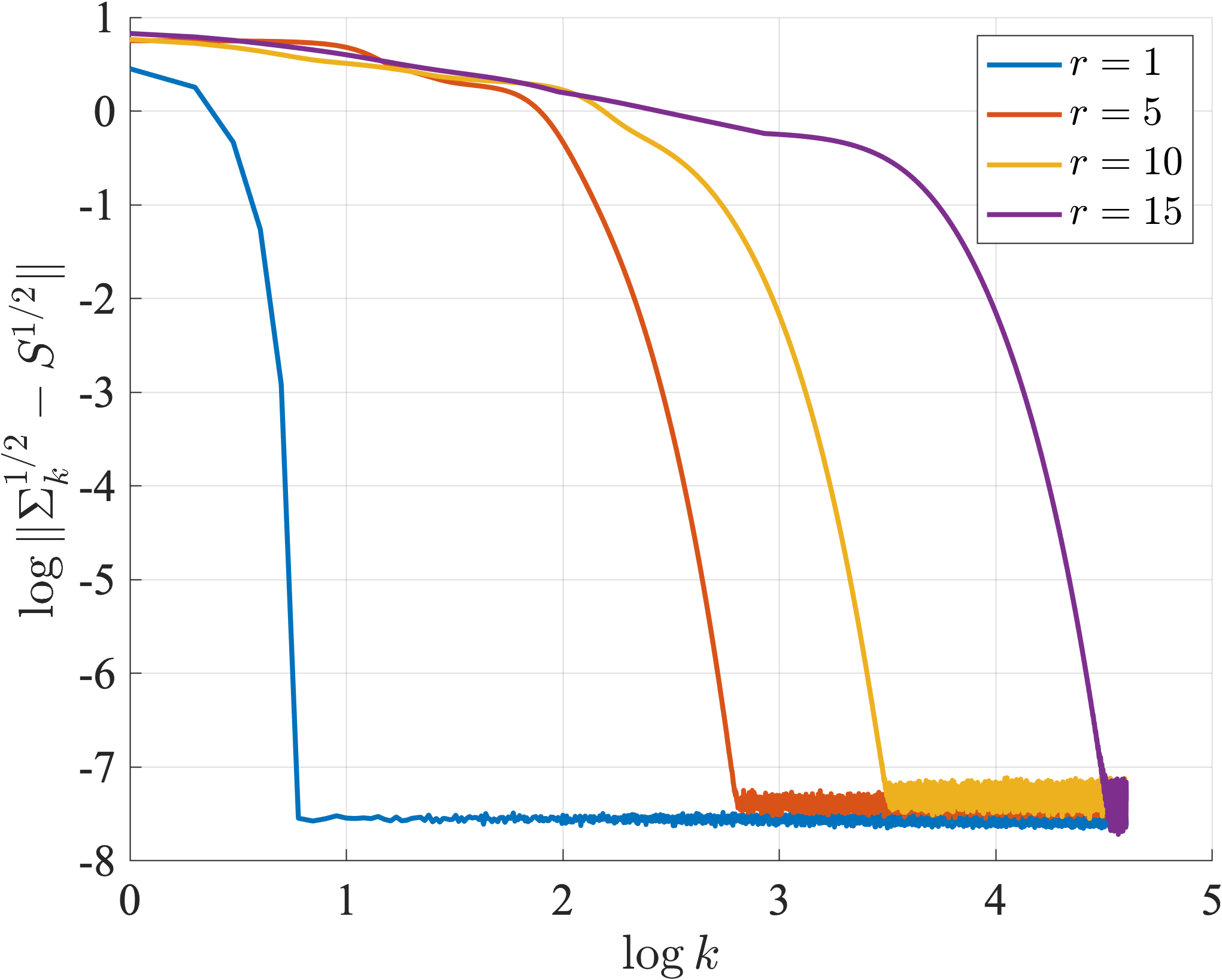}
    \includegraphics[width = 0.40\textwidth]{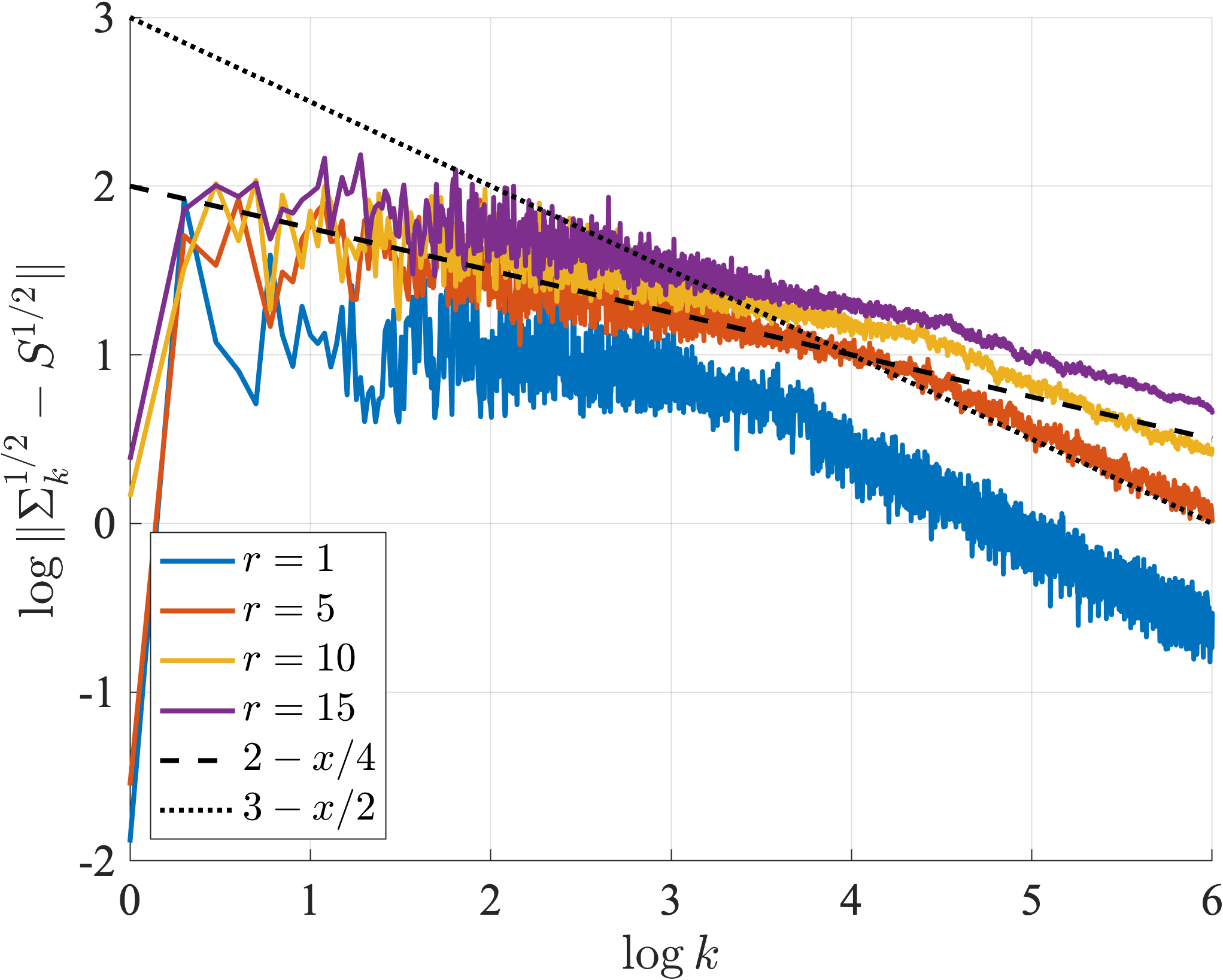}
    \caption{Convergence of BWGD and BWSGD for varying ranks.}
    \label{fig:ranks}
\end{figure}

In the plot for SGD, we also give lines to show the different rates. Here, it appears that SGD is converging to the true barycenter. Also, it appears to be converging at a faster than anticipated rate. An explanation of this phenomenon will be explored in future work. In the left plot of Figure \ref{fig:ranks}, we set $d=20$ and $n=d^2$, and we plot the error versus iteration for \BW\ gradient descent for the various ranks. As we can see, the convergence takes longer as the rank increases.
In the right plot of Figure \ref{fig:ranks}, we plot the error versus iteration for \BW SGD using the single sample gradient of \eqref{eq:gradss}, where at each iteration we draw a new sample. As we can see, the \BW SGD interpolates between two convergence regimes: a slow regime where the rate is $k^{-1/4}$ and a fast regime where the rate is $k^{-1/2}$. The latter rate is typical of cases where there is local strong convexity or a Polyak-\L{}ojasiewicz inequality.

\subsection{Abalone Dataset}

\begin{figure}[h!]
    \centering
    \includegraphics[width = 0.8\textwidth]{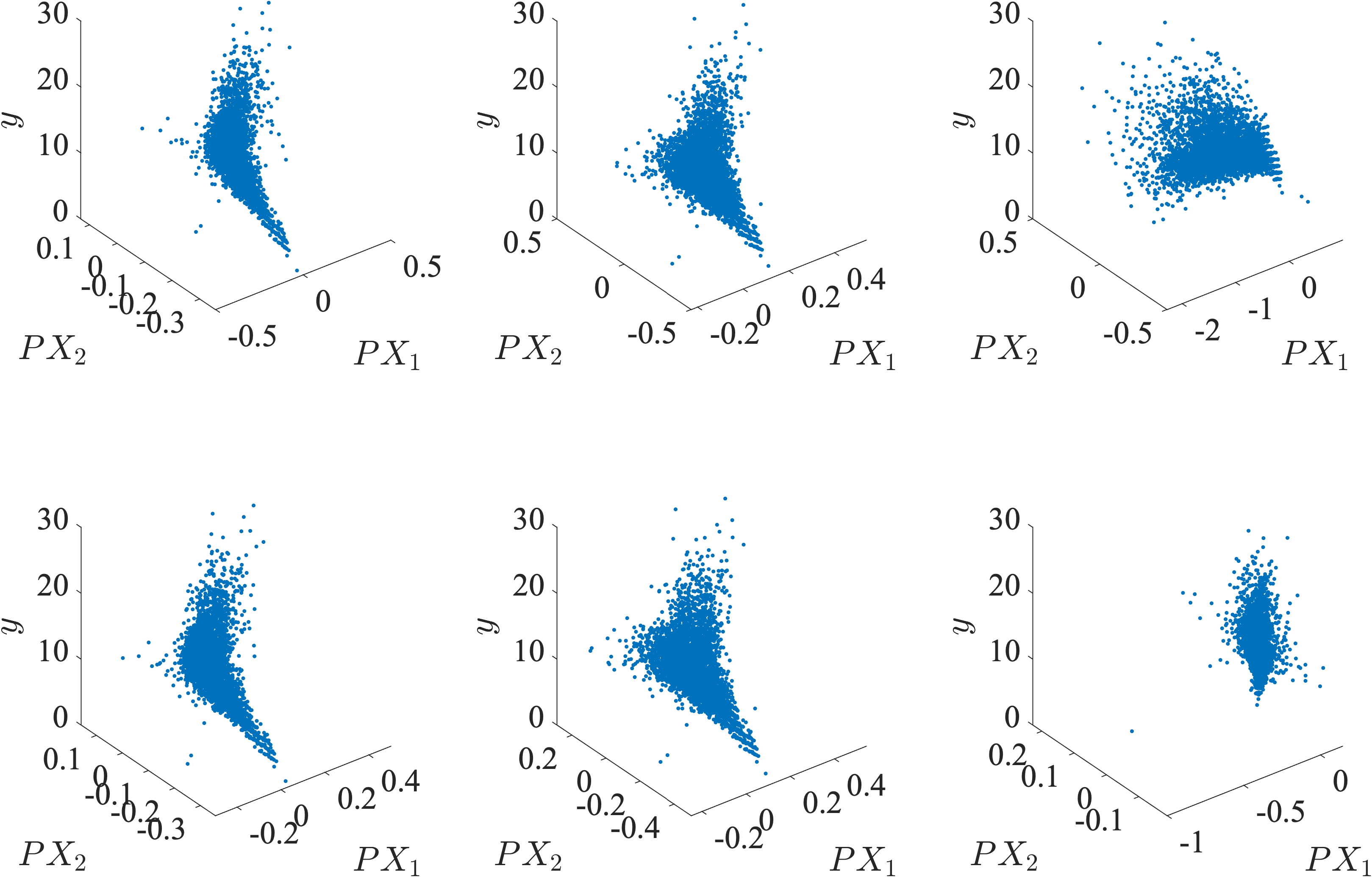}
    \caption{Projections for abalone data. The left column are projections for a fixed rank \BW\ and full-rank \BW\ barycenter, the middle column corresponds to low-rank and full-rank Euclidean gradient descent on \eqref{eq:lsmatrec_nc1}, and the right column corresponds to the top principal subspace of $\bx$ and the result of the spectral method.}
    \label{fig:aba}
\end{figure}

We also present an experiment with real data. This example is one where we attempt to measure the heterogeneity in a regression dataset. Here, we pick the classical Abalone dataset available from the UCI machine learning repository~\citep{uciml}. In this dataset, one attempts to predict the age of abalone from certain covariates. Linear regression models exhibit poor fit on this data for a variety of reasons. One reason, which we demonstrate here, is the presence of heterogeneity in the measurements.

Assuming the covariance recovery model $y_i = \langle \bx_i, \bw_i\rangle + \epsilon_i,$
 where $\bw_i \sim N(\bzero, \bS)$, we could try to measure heterogeneity in the data by recovering the covariance of the regression vectors, and then plotting how $\bx$ relates to $y$ in its top principal space. Here, we recover such a covariance, and then project $\bx$ onto the top two principal directions. We then plot $y$ against these two directions. The resulting plots show varying degrees of heterogenity depending on the method employed. Here, we compare \BW\ gradient descent, GD, PCA on the $\bx$'s, and the spectral method, which finds the top principal directions of the matrix 
\begin{equation}
	S_n = \frac{1}{2n}\sum_{i=1}^n y_i (\bx_i\bx_i^\top - \bI)
\end{equation}

As we can see, the spectral methods completely fails here. The \BW\ gradient descent method gives a similar but qualitatively different result from the Euclidean gradient descent method. In particular, there appear to be two primary directions of variation, which would indicate that this may be a mixture of two different regression components. Both \BW\ gradient descent and GD recover a stretched direction, but the secondary direction (in the $PX_1$ direction) appears to be more pronounced for \BW\ gradient descent.

\end{document}